\documentclass[reqno]{amsart}
\usepackage{amssymb,amscd,verbatim, amsthm,graphics, color,latexsym,amsmath,multicol}
\usepackage{fancybox}
\usepackage{longtable}

\usepackage[all]{xy}

\usepackage[top=2.7cm, bottom=2.7cm, left=2.7cm, right=2.7cm]{geometry}

\newcommand \C[1]{{\mathcal #1}}

\newcommand \wti[1]{{\widetilde {#1}}}

\newcommand \bC{{\mathbb C}}

\newcommand \bH{{\mathbb H}}

\newcommand \bZ{{\mathbb Z}}

\newcommand \bQ{{\mathbb Q}}

\newcommand\CF{{\C F}}
\newcommand\CH{{\C H}}

\newcommand\CS{{\C S}}

\newcommand\CX{{\C X}}

\newcommand\CR{{\C R}}

\newcommand\CY{{\C Y}}

\newtheorem{theorem}{Theorem}[section]
\newtheorem{conjecture}[theorem]{Conjecture}
\newtheorem{corollary}[theorem]{Corollary}
\newtheorem{lemma}[theorem]{Lemma}
\newtheorem{proposition}[theorem]{Proposition}
\newtheorem{definition}[theorem]{Definition}

\newtheorem{example}[theorem]{Example}

\newcommand\Ind{\operatorname{Ind}}

\newcommand\sgn{\mathsf{sgn}}

\newcommand\sem{\mathsf{ss}}
\newcommand\Irr{\mathsf{Irr}}

\newcommand\el{\mathsf{ell}}
\newcommand\res{\mathsf{res}}

\newcommand\un{\mathsf{un}}
\newcommand\pa{\mathsf{par}}

\newcommand\Id{\operatorname{Id}}

\newcommand\ad{\operatorname{ad}}

\newcommand\proj{\operatorname{proj}}
\newcommand\FT{\operatorname{FT}}

\def\<{\langle} 
\def\>{\rangle}

\numberwithin{equation}{section}

\begin{document}

\title[The elliptic nonabelian Fourier transform for exceptional $p$-adic groups]{The nonabelian Fourier transform for elliptic unipotent representations of exceptional $p$-adic groups}

\author
{Dan Ciubotaru}
        \address[D. Ciubotaru]{Mathematical Institute, University of Oxford, Oxford OX2 6GG, UK}
        \email{dan.ciubotaru@maths.ox.ac.uk}

\begin{abstract} We define an involution on the space of elliptic unipotent Langlands parameters of a reductive $p$-adic group $G$ and verify that when $G$ is split adjoint exceptional, the composition of this involution with the hyperspecial parahoric restriction map agrees with  Lusztig's nonabelian Fourier transform for unipotent representations of the finite reductive  quotient. This is inspired by recent works of Lusztig on the almost unipotent characters of $p$-adic groups and of M\oe glin and Waldspurger  on the elliptic Fourier transform of odd orthogonal groups.
\end{abstract}




\maketitle

\setcounter{tocdepth}{1}
\tableofcontents

\section{Introduction}  Let $\mathsf k$ be a nonarchimedean local field with residual cardinality $q$. Let $G$ be the $\mathsf k$-points of a connected simple adjoint group defined over $\mathsf k$. If $K$ is a parahoric subgroup of $G$, denote by $U_K$ its pro-unipotent radical such that $K/U_K$ is a finite reductive group. An irreducible smooth representation $(\pi,V)$ of $G$ is called unipotent \cite{L1} if there exists a parahoric subgroup $K$ such that the invariants $V^{U_K}$ contain an irreducible unipotent representation of $K/U_K$. Let $\Irr_\un G$ denote the set of isomorphism classes of irreducible unipotent $G$-representations and let $R_\un(G)$ denote the $\bC$-span  of $\Irr_\un G$ (i.e., the complexification of the Grothendieck group of the abelian category of smooth unipotent $G$-representations). Let $G^\vee$ be the complex dual Langlands group with centre $Z$.  If $x\in G^\vee$, let $Z_{G^\vee}(x)$ denote the centraliser of $x$ in $G^\vee$ and $A_x=Z_{G^\vee}(x)/Z_{G^\vee}(x)^0$ the finite group of components. The results of \cite{L1}, which extend the Iwahori-spherical correspondence of \cite{KL}, give a natural one-to-one correspondence 
\begin{equation}\label{e:DLL}
 G^\vee\backslash \{(x,\phi)\mid x\in G^\vee,\ \phi\in \widehat A_x\}\longleftrightarrow\bigsqcup_{G'}~\Irr_\un G',\quad (x,\phi)\mapsto (\pi_{(x,\phi)},V_{(x,\phi)}),
\end{equation}
where the union in the right hand side is over the groups $G'$ which are in the same inner class as $G$. In this correspondence, $\Irr_\un G$ are parameterised by $(x,\phi)$ such that $\phi |_Z=\Id$.

Let $\overline R_\un(G)$ be the elliptic representation space, which is the quotient of $R_\un(G)$ by the subspace spanned by all properly parabolically induced representations. This is endowed with a nondegenerate hermitian pairing, called the Euler-Poincar\'e pairing, see \cite{SS, Re1}. The elliptic space has a natural parameterisation in terms of Langlands parameters \cite{Re1} as follows. For $x\in G^\vee$, fix $T_x$ a maximal connected torus in  $Z_{G^\vee}(x)^0$ and let $\mathfrak t_x$ be the Lie algebra of $T_x$. The group $A_x$ acts of $T_x$ and on $\mathfrak t_x$, and let $\overline R(A_x)$ denote the corresponding elliptic representation complex space endowed with the (nondegenerate) elliptic pairing and $A_{x,\el}$ the set of elliptic elements of $A_x$. Then the parameterisation (\ref{e:DLL}) induces a isometric isomorphism 
\begin{equation}
\bigoplus_{G'} \overline R_\un (G') \longrightarrow \CX(G^\vee)_\el:=\bigoplus_{x\in G^\vee} \overline R(A_x),
\end{equation}
where the sum in the right hand side is over conjugacy classes in $G^\vee$. If one restricts to $ \overline R_\un (G)$ in the left hand side, then the corresponding space in the right hand side is $\CX(G^\vee)_\el^0:=\bigoplus_{x\in G^\vee} \overline R(A_x^{\ad})$, where $A_x^{\ad}=Z_{G^\vee}(x)/Z_{G^\vee}(x)^0 Z$. This was proved by Reeder \cite{Re1} in the case of Iwahori-spherical $G$-representations, and the results of  Waldspurger \cite{Wa2} allow one to generalise the isomorphism to unipotent representations. 

The space $ \overline R(A_x)$ is zero unless $x=su$ (Jordan decomposition) is such that the semisimple element $s$ is isolated, meaning that $Z_{G^\vee}(s)$ is a semisimple group, and $u$ is quasi-distinguished in $Z_{G^\vee}(s)$ (in the sense of \cite{Re1}). Consequently, we can reorganise the elliptic parameter space as 
\begin{equation}\label{decomp}
\CX(G^\vee)_\el=\bigoplus_{u\in G^\vee,\text{ unipotent} } \CX(G^\vee)_\el^u,\quad \CX(G^\vee)_\el^u=\bigoplus_{s\in Z_{G^\vee}(u),\text{ semisimple}} \overline R(A_{su}).
\end{equation}
The first sum is over $G^\vee$-conjugacy classes of unipotent elements $u\in G^\vee$ and the second sum is over $Z_{G^\vee}(u)$-conjugacy classes of semisimple elements $s\in Z_{G^\vee}(u)$. 

\medskip

The following construction appears in \cite{Wa1}, when $G^\vee$ is the symplectic group (and necessarily the semisimple elements below have order at most $2$), but the generalisation is straight-forward. Denote
\begin{equation}
\Gamma_u=\text{ the reductive part of } Z_{G^\vee}(u).
\end{equation}
\begin{definition} 
Define the set 
\[\C Y(\Gamma_u)_\el=\{(s,h)\in \Gamma_u\times \Gamma_u\mid s,h\text{ semisimple in }G^\vee,\ sh=hs,\ Z_{\Gamma_u}(s,h)\text{ is finite}\}.
\]
Since $s,h$ are commuting semisimple elements, $Z_{\Gamma_u}(s,h)=Z_{\Gamma_u}(s)\cap Z_{\Gamma_u}(h)$ is a reductive group, so the condition is equivalent to saying that there is no nontrivial torus in $\Gamma_u$ centralised by both $s$ and $h$.
\end{definition}

Using (\ref{decomp}), it is easy to see that we can naturally identify $\CX(G^\vee)^u_\el$ with $\Gamma_u$-equivariant functions on $\C Y(\Gamma_u)_\el$:
\begin{equation}
\CX(G^\vee)^u_\el=\bC[\C Y(\Gamma_u)_\el]^{\Gamma_u}=\bC[\Gamma_u\backslash \C Y(\Gamma_u)_\el].
\end{equation}

\

For every pair $(s,h)\in \Gamma_u\backslash \C Y(\Gamma_u)_\el$, define the virtual representation:
\begin{equation}
\pi(u,s,h)=\sum_{\phi\in \widehat A_{su}} \overline{\phi(h)}~ \pi(su,\phi).
\end{equation}
These are virtual elliptic tempered representations in the sense of Arthur \cite{Ar} and they (or rather their images) form an orthogonal basis of $\bigoplus_{G'} \overline R_\un (G')$. This allows us to identify $\bigoplus_{G'} \overline R_\un (G')$ with the {\it subspace} of $\bigoplus_{G'} R_\un(G')$ spanned by the virtual character combinations $\pi(u,s,h)$, which we denote
\begin{equation}
\CR_{\un,\el}\cong \bigoplus_{G'} \overline R_\un (G').
\end{equation}

 Let $\pi(u,s,h)_0$ denote the projection of $\pi(u,s,h)$ onto $\overline R_\un(G)$, let \[\CX(G^\vee)^{0,u}_\el=\bC[\C Y(\Gamma_u)_\el]^{\Gamma_u}_0\] be the corresponding elliptic Langlands parameter subspace and $R_\un(G)_\el$ the subspace of $\CR_{\un,\el}$ spanned by $\{\pi(u,s,h)_0\}$. 

\begin{definition}[cf. \cite{Wa1}] The dual \emph{elliptic nonabelian Fourier transform} is the linear map
\[
\FT^\vee_\el: \CR_{\un,\el}\to \CR_{\un,\el},\quad \FT^\vee_\el(\pi(u,s,h))=\pi(u,h,s), \ (s,h)\in \Gamma_u\backslash \C Y(\Gamma_u)_\el,\ u\in G^\vee \text{ unipotent}.
\]
\end{definition}
When $u$ is a distinguished unipotent element, in which case, $\Gamma_u=A_u$ is a finite group, this definition agrees with the one considered in \cite{CO2}. 

\

We are interested in the interaction between  $\FT^\vee_\el$ and parahoric restrictions. Define the parahoric restriction as a linear map onto unipotent class functions of the reductive quotient of a parahoric subgroup: 
\begin{equation}
\res^K_\un: R_\un(G)\to \bC_\un[K/U_K]^{K/U_K},\quad V\mapsto V^{U_K},\text{ for all } V\in\Irr_\un G.
\end{equation} 
Let $\bC_\un[K/U_K]^{K/U_K}_\el$ denote the elliptic subspace of $\bC_\un[K/U_K]^{K/U_K}$, i.e., the unipotent class functions on $K/U_K$ which are annihilated by all proper Jacquet restriction functors of $K/U_K$. Let 
\[\proj^K_\el: \bC_\un[K/U_K]^{K/U_K}\to \bC_\un[K/U_K]^{K/U_K}_\el
\]
denote the projection map with respect to the ordinary character pairing. Set 
\[\res^\pa_\un=(\res^K_\un)_K,\quad  \proj^\pa_\el=(\proj^K_\el)_K,\]
 where $K$ ranges over the conjugacy classes of {\it maximal} parahoric subgroups of $G$. A similar definition applies to every $G'$ as well.

\medskip

The space $\bC_\un[K/U_K]^{K/U_K}$ has a basis given by the irreducible unipotent characters and another basis given by the almost characters \cite{orange}. The change of basis matrix gives the nonabelian Fourier transform for finite Lie groups, which we denote by
\begin{equation}
\FT^K:\bC_\un[K/U_K]^{K/U_K}\to \bC_\un[K/U_K]^{K/U_K} \text{ and } \FT^\pa=(\FT^K)_{K \text{ maximal}}.
\end{equation}
Explicitly, $\FT^K$ is given as follows. Lusztig partitioned $\Irr_\un K/U_K$ into families $\C F$ and each family has a finite group $\Gamma_{\C F}$ attached such that the elements of $\C F$ are in one-to-one correspondence with irreducible $\Gamma_{\C F}$-equivariant local systems on $\Gamma_{\C F}$, i.e., with 
\[
\Gamma_\CF\backslash\{(x,\rho)\mid \rho\in\widehat Z_{\Gamma_\CF}(x)\},\quad (x,\rho)\mapsto \sigma(x,\rho)\in \CF.
\]
 If we denote by $R_\un(K/U_K)_\CF$ the $\bC$-span of $\CF$, then an equivalent interpretation is that 
\[R_\un(K/U_K)_\CF\cong \bC[\CY(\Gamma_\CF)]^{\Gamma_\CF}.
\]
Here $\CY(\Gamma_\CF)=\{(x,y)\mid x,y\in\Gamma_{\CF},\ xy=yx\}$ and of course $\CY(\Gamma_\CF)=\CY(\Gamma_\CF)_\el$ in our previous notation since $\Gamma_\CF$ is finite. If we define the virtual characters
\begin{equation}
\sigma(\CF, x,y)=\sum_{\rho\in \widehat Z_{\Gamma_\CF}(x)} \overline{\rho(y)}~ \sigma(x,\rho),\quad (x,y)\in\CY(\Gamma_\CF),
\end{equation}
then the involution $\FT^K$ is defined by the requirement that 
\begin{equation}
\FT^K(\sigma(\CF,x,y))=\Delta(x,y)~ \sigma(\CF,y,x), 
\end{equation}
where $\Delta(x,y)=1$ in all cases, except when the family $\CF$ contains the characters $\phi_{512,11}$ in $E_7$, $\phi_{4096,11}$ or $\phi_{4096,26}$ in $E_8$. For these exceptions, the finite group is $\Gamma_{\CF}=C_2=\{1,g_2\}$, and one sets $\Delta(1,g_2)=\Delta(g_2,1)=-1.$

\

We expect the following relation.

\begin{conjecture}[{cf. \cite{MW,Wa1}}]\label{conj} With the notation as above:

\begin{enumerate}
\item The two Fourier transforms commute on the elliptic space $\CR_{\un,\el}$, 
\[\proj^\pa_\el\circ ~\res^\pa_\un\circ \FT^\vee_\el= \proj^\pa_\el\circ \FT^\pa\circ ~\res^\pa.\]
\item If $K_0$ is the maximal hyperspecial parahoric of $G$, then on the elliptic subspace $R_\un(G)_\el$ for $G$,
\[\res^{K_0}_\un\circ \FT^\vee_\el=\FT^{K_0}\circ~ \res^{K_0}_\un.
\]
\end{enumerate}
\end{conjecture}
The point of (2) is that the identity should hold for $K_0$ {\it without} projecting onto the elliptic spaces of the finite parahoric quotients. 

\smallskip

In \cite{MW}, M\oe glin and Waldspurger prove Conjecture \ref{conj}(1) when $G^\vee$ is the symplectic group and therefore $G$ is the split odd orthogonal $p$-adic group (together with its inner form). In \cite{Wa1}, Waldspurger simplifies the arguments and shows that (2) also holds true {\it for all} maximal parahoric subgroups of the odd orthogonal groups, not only $K_0$. The expected compatibility of the two Fourier transforms above is also motivated by the geometric programme of constructing unipotent almost characters for reductive $p$-adic groups as in the work of Lusztig \cite{L4,L5} and Bezrukavnikov-Kazhdan-Varshavsky \cite{BKV}.

\medskip

In the rest of the paper, we verify Conjecture \ref{conj}(2) when $G$ is a split adjoint exceptional group. 
To state a more precise version of our result, let $u$ be a special unipotent class (in the sense of \cite{orange}) and let $\tau(u,\mathbf 1)$ denote the $G(\mathsf F_q)$-type parameterised via the Springer correspondence by the sign-dual of the irreducible Weyl group representation in the top cohomology for the trivial local system supported on $G^\vee\cdot u$. Let $\CF_u$ denote the unipotent $G(\mathsf F_q)$-family whose special element is $\tau(u,\mathbf 1)$. The finite group $\Gamma_{\CF_u}$ associated to $\CF_u$ is Lusztig's canonical quotient $\overline A_u$ of $A_u$. Let $b_\CF$ denote the lowest harmonic degree of the special representation in the family $\CF$. 

The quotient map $\Gamma_u\twoheadrightarrow A_u$ gives rise to a surjection $(s,h)\mapsto (\bar s,\bar h)$,
\begin{equation}
 \CY(\Gamma_u)_\el\longrightarrow \CY(A_u)_\el=\{(\bar s,\bar h)\in A_{u,\el}\times A_{u,\el}\mid \bar s\bar h=\bar h \bar s \}.
\end{equation}

\medskip

We first determine the list of unipotent classes $u$ in a simply-connected group $G^\vee$ of exceptional type which have the property that $\CY(\Gamma_u)_\el\neq \emptyset$. It turns out that all of these unipotent classes are in particular special. The list  is given in Proposition \ref{p:classes}. Then the main result is:

\begin{theorem}\label{t:main}
Suppose $G$ is split adjoint of type $G_2$, $F_4$, $E_6$, $E_7$, or $E_8$. For every unipotent class $u$ with $\CY(\Gamma_u)_\el\neq \emptyset$ and every $(s,h)\in \Gamma_u\backslash \CY(\Gamma_u)_\el$:
\begin{equation}\label{e:restrict}
\res^{K_0}_\un (\pi(u,s,h)_0)=\zeta(s,h)~\sigma(\CF_u,\bar s,\bar h)+\sum_{\substack{\CF\subset \Irr_u G(\mathsf F_q)\\b_{\CF_u}<b_{\CF}}} \ \sum_{(x,y)\in \Gamma_\CF\backslash\CY(\Gamma_\CF)} a_{(x,y)}^{\CF}(s,h) ~\sigma(\CF,x,y),
\end{equation}
for a root of unity $\zeta(s,h)$ and constants $a_{(x,y)}^{\CF}(s,h)\in\overline\bQ$. Moreover: 
\begin{enumerate}
\item  $\zeta(s,h)~\overline{\zeta(h,s)}=\Delta(\bar s,\bar h)$,
\smallskip
\item $a_{(x,y)}^{\CF}(h,s)=\Delta(x,y)~a_{(y,x)}^{\CF}(s,h)$,  for all  $(x,y)\in\Gamma_\CF\backslash\CY(\Gamma_\CF)$ and $(s,h)\in \Gamma_u\backslash\CY(\Gamma_u)_\el$.
\end{enumerate}
\end{theorem}

In particular, this says that
\begin{equation}\label{comm}
\res^{K_0}_\un (\pi(u,h,s)_0)=\FT^{K_0}\circ ~\res^{K_0}_\un (\pi(u,s,h)_0),
\end{equation}
for all unipotent elements $u$ such that $(s,h)\in \Gamma_u\backslash\CY(\Gamma_u)_\el\neq \emptyset.$ 
Since the set $\{\pi(u,s,h)_0\}_{(u,s,h)}$ is a basis of the elliptic representation space $R_\un(G)_\el$, the compatibility of the two Fourier transforms on the elliptic unipotent representation space follows at once from (\ref{comm}). 

\begin{corollary}
Conjecture \ref{conj}(2)  holds true when $G$ is a split adjoint exceptional group.
\end{corollary}

We prove Theorem \ref{t:main} by direct calculations, but with this interpretation of $\FT^\vee_\el$ and the available literature, the calculations are transparent and can be done by hand. We rely on Reeder's \cite{Re2} calculations and tables for unipotent discrete series and on the tables in Carter's book \cite{Ca}, Beynon-Spaltenstein's calculation of Springer representations \cite{BS} and Alvis's tables of restrictions for Weyl group representations \cite{Al}. In addition, we need the $W$-structure in some instances of Springer's cohomological representations for classical groups (such as $D_8$), and in our particular cases, it is fast to compute this via the different algorithm of \cite{CKK}. There are also several cases of elliptic, but not discrete series, representations which are not available in the literature and we compute the restrictions in this paper, using Hecke algebra methods.  The description of the reductive part $\Gamma_u$ of the centralisers of unipotent elements has been known in large part for a long time, but for our calculations we need not only the description of the connected component $\Gamma_u^0$ and the component group $A_u$ (which are available in \cite{Ca} for example), but also precise information about the extension $1\to \Gamma_u^0\to\Gamma_u\to A_u\to 1$. Complete descriptions for the exceptional groups are now available in \cite{Ad} and we used them in a handful of interesting cases.  

\smallskip

Particularly interesting instances of the relation in the theorem are provided by the non-distinguished cases, see for example  $u=A_4+A_1$ in $E_7$ (section \ref{s:A4A1}) or $A_4+2A_1$ in $E_8$ (section \ref{s:A42A1}).

\smallskip

It appears likely that Conjecture \ref{conj}(2) holds in fact for all maximal parahoric subgroups in the appropriate formulation.\footnote{In a previous version of this paper, I included a calculation for the parahoric restrictions of type $A_2\times \wti A_2$ in $F_4$ which seemed to indicate that Conjecture \ref{conj}(2) does not hold in that case. The calculation was based on the restriction tables  \cite[p. 65-66]{Re2}, which turn out to be incorrect in the cases of $[F_4(a_3),(31)]$ and $[F_4(a_3),(211)]$. Once the corrections are made, the expected identity of parahoric restrictions can be verified in this example as well.} We will return to this question in future work. 

\

\noindent{\bf Acknowledgements.} This research was supported in part by the EPSRC grant EP/N033922/1.
This work builds on ideas from joint work with Eric Opdam \cite{CO2}.  I thank Jeff Adams for sharing his calculations of the centralisers $\Gamma_u$ and Anne-Marie Aubert and Eric Opdam for very helpful discussions about the nonabelian Fourier transform. I am indebted to Jean-Loup Waldspurger for pointing out the error in a previous version of this paper in the calculation of the parahoric restriction of type $A_2\times \wti A_2$ of the virtual combinations for $F_4$ and for a very valuable correspondence.

\section{Centralisers and component groups}

\subsection{The elliptic spaces $\overline R(A_x)$} Let $G^\vee$ be a complex connected simply-connected semisimple group. Suppose $x\in G^\vee$ has a Jordan decomposition $x=su=us$, where $s$ is semisimple and $u$ is unipotent. The centraliser $Z_{G^\vee}(s)$ is a connected reductive group and $u\in Z_{G^\vee}(s)$. Let $\Gamma_x$ be the reductive part of the centraliser of $u$ in $Z_{G^\vee}(s)$ and let $A_x=\Gamma_x/\Gamma_x^0$ be the group of components. Let $T_x$ be a maximal torus in $\Gamma_x$ and $B_x\supset T_x$ a Borel subgroup of $\Gamma_x$. As in \cite[3.2]{Re1}, there is a natural isomorphism 
\[A_x\cong N(B_x,T_x)/{T_x},
\]
via which $A_x$ acts on $\mathfrak t_x$, the Lie algebra of $T_x$. The elliptic character pairing on $R(A_x)$ is defined as
\begin{equation}
(\psi,\psi')^\el_{A_x}=\frac 1{|A_x|}\sum_{\bar h\in A_x} \overline\psi(\bar h) \psi'(\bar h) {\det}_{\mathfrak t_x}(1-\bar h).
\end{equation}
Element $\bar h\in A_x$ is called elliptic if $ {\det}_{\mathfrak t_x}(1-\bar h)\neq 0$. Let $A_{x,\el}$ denote the set of elliptic elements (a uniom of conjugacy classes). If $L$ is a proper Levi subgroup of $Z_{G^\vee}(s)$ containing $T_x$ and such that $u\in L$, denote by $A_x^L$ the group of components of the centraliser of $u$ in $L$. By \cite[Lemma 3.2.1]{Re1}, the inclusion $L\to Z_{G^\vee}(s)$ induces an inclusion $A_x^L\hookrightarrow A_x$ such that
\begin{equation}\label{A-elliptic}
A_{x,\el}=A_x\setminus\bigcup_{L\in \C L_x} A_x^L,
\end{equation}
where $\C L_x$ is the set of such Levi subgroups of $Z_{G^\vee}(s)$.  As in \cite[Lemma 3.2.2]{Re1}, it follows that $A_{x,\el}\neq\emptyset$ if and only if there exists $h\in \Gamma_x$ such that $hu\notin L$ for any $L\in \C L_x$. One says that $u$ is {\it quasi-distinguished} in $Z_{G^\vee}(s)$ in this case, and this implies in particular that $Z_{G^\vee}(s)$ is a semisimple group, i.e., that $s$ is {isolated}.

Let $\ker(~,~)_{A_x}^\el$ be the radical of the elliptic character pairing and define $\overline R(A_x)$ to the quotient of $R(A_x)$ by it. This can be naturally identified \cite[Proposition 2.2.2]{Re1} with the space $\bC[A_{x,\el}]^{A_x}$ of class functions on $A_x$ supported on $A_{x,\el}$. In fact, the restriction of characters map
\begin{equation}
r: \overline R(A_x)\longrightarrow \bC[A_{x,\el}]^{A_x}
\end{equation}
is an isometry if we endow the target with the pairing
\[(f,f')=\frac 1{|A_x|}\sum_{\bar h\in A_{x,\el}}  \overline {f(\bar h)} f'(\bar h) {\det}_{\mathfrak t_x}(1-\bar h).
\]

Let $\C C(\Gamma_u)_{\sem}$ denote the set of $\Gamma_u$-orbits on $\{s\in G^\vee\mid s\text{ semisimple}\ su=us\}$ and $\C C(A_{x,\el})$ the set of elliptic conjugacy classes in $A_x$.
\begin{lemma}
For every unipotent element $u\in G^\vee$, there is a natural one-to-one correspondence 
\[ \bigsqcup_{s\in \C C(\Gamma_u)_{\sem}} \C C(A_{su,\el})\longleftrightarrow \Gamma_u\backslash \CY(\Gamma_u)_\el.
\]
\end{lemma}

\begin{proof}
This is immediate from the definitions. Fix $s\in \C C(\Gamma_u)_{\sem}$. Every semisimple element $\bar h\in A_{su}$ can be represented by a semisimple element $h\in Z_{\Gamma_u}(s)$, since the unipotent elements of $Z_{\Gamma_u}(s)$ lie in $Z_{\Gamma_u}(s)^0$. Moreover, if $h,h'$ are in the same conjugacy class of semisimple elements in $Z_{\Gamma_u}(s)$, then $\bar h$, $\bar {h'}$ are in the same conjugacy class in $A_{su}$. By (\ref{A-elliptic}), $\bar h$ is elliptic in $A_{su}$ if and only if $h$ is not contained in any proper Levi subgroup $L$ of $Z_{G^\vee}(s)$ such that $u\in L$. This is equivalent with the requirement that there is no nontrivial torus (the connected centre of $L$) in $Z_{G^\vee}(s)$ such that both $u$ and $h$ centralise it. Equivalently, there is no nontrivial torus in $\Gamma_u$ which is centralised by both $s$ and $h$, which by definition means $(s,h)\in \CY(\Gamma_u)_\el$. 
\end{proof}

This means we may identify, as mentioned in the introduction, 
\begin{equation}
\bigoplus_{s\in \C C(\Gamma_u)_{\sem}}\overline R(A_x)\cong \bC[\CY(\Gamma_u)_\el]^{\Gamma_u}.
\end{equation}

\subsection{Centralisers $\Gamma_u$} We list the possibilities for the centralisers $\Gamma_u$, spaces $\CY(\Gamma_u)_\el$, and virtual combinations $\pi(u,s,h)$ that occur for adjoint exceptional groups. An easy observation:

\begin{lemma}\label{l:sc}
If $\Gamma_u$ is a connected, simply-connected semisimple group, then $\CY(\Gamma_u)_\el=\emptyset$.
\end{lemma}

\begin{proof}
Let $s\in \Gamma_u$ be a semisimple element. Then $Z_{\Gamma_u}(s)$ is a nontrivial connected reductive group, since $\Gamma_u$ is connected and simply-connected. If $h$ is semisimple in $Z_{\Gamma_u}(s)$, $h$ is contained in a nontrivial torus of $Z_{\Gamma_u}(s)$, hence $Z_{\Gamma_u}(s)\cap Z_{\Gamma_u}(h)$ is not finite. 
\end{proof}

\begin{example}
If $\Gamma=\mathsf{PGL}(n)$, then $\CY(\Gamma)_\el\neq \emptyset$. If $s$ is the diagonal element $s=(1,\zeta,\zeta^2,\dots,\zeta^n)$, where $\zeta$ is a primitive $n$-th root of $1$, then $Z_\Gamma(s)=T\rtimes C_n$, where $T=(\bC^\times)^n/\bC^\times I_n$ is the maximal diagonal torus and $C_n=\langle w\rangle$, where $w$ is the permutation matrix corresponding to an $n$-cycle in $S_n$. Hence $T^{w}=\langle s\rangle\cong C_n$, and $(s,w)\in \CY(\Gamma)_\el$. This case does not occur for exceptional adjoint $p$-adic groups, except when $n=2,3$. But it appears, for example, for $\mathsf{SL}(n,\bQ_p)$ and the trivial unipotent class. 
\end{example}

\begin{proposition}\label{p:classes}
The lists of unipotent classes $u$ in the simply-connected exceptional groups such that $\CY(\Gamma_u)_\el\neq \emptyset$ are given in Tables \ref{t:G2}, \ref{t:F4}, \ref{t:E6}, \ref{t:E7}, and \ref{t:E8}.
\end{proposition}

\begin{proof}
The proof is by inspection. Using Lemma \ref{l:sc}, we eliminate many cases. In addition, we know that in order for $\CY(\Gamma_u)_\el$ to be non-empty, $u$ must be a quasi-distinguished unipotent class in one of the maximal pseudo-Levi subgroups $Z_{G^\vee}(s)$ of $G^\vee$ and the quasi-distinguished classes are known by \cite{Re2}, also \cite{CH}. The cases that appear in this way are discussed in the rest of the section.
\end{proof}

\begin{table}[h]
\caption{Elliptic pairs for $G_2$\label{t:G2}}
\begin{tabular}{|c|c|c|c|c|c|}
\hline
Unipotent &$\Gamma_u^0$ &$A_u$ &$|\Gamma_u\backslash \CY(\Gamma_u)_\el|$ &$\CF_u$ &$\Gamma_{\CF_u}$\\
\hline
\hline
$G_2$ &$1$ &$1$ &$1$ &$\phi_{(1,6)}$ &$1$\\
\hline
$G_2(a_1)$ &$1$ &$S_3$ &$8$ &$\phi_{(2,1)}$ &$S_3$\\
\hline
\end{tabular}
\end{table}

\begin{table}[h]
\caption{Elliptic pairs for $F_4$\label{t:F4}}
\begin{tabular}{|c|c|c|c|c|c|}
\hline
Unipotent &$\Gamma_u^0$ &$A_u$ &$|\Gamma_u\backslash \CY(\Gamma_u)_\el|$ &$\CF_u$ &$\Gamma_{\CF_u}$\\
\hline
\hline
$F_4$ &$1$ &$1$ &$1$ &$\phi_{1,24}$ &$1$\\
\hline
$F_4(a_1)$ &$1$ &$S_2$ &$4$ &$\phi_{4,13}$ &$C_2$\\
\hline
$F_4(a_2)$ &$1$ &$S_2$ &$4$ &$\phi_{9,10}$ &$1$\\
\hline
$B_3$ &$\mathsf{PGL}(2)$ &$1$ &$1$ &$\phi_{8,9}''$ &$1$\\
\hline
$F_4(a_3)$ &$1$ &$S_4$ &$21$ &$\phi_{12,4}$ &$S_4$\\
\hline
\end{tabular}
\end{table}

\begin{table}[h]
\caption{Elliptic pairs for $E_6$\label{t:E6}}
\begin{tabular}{|c|c|c|c|c|c|}
\hline
Unipotent &$\Gamma_u^0$ &$A_u^{\ad}$ &$|\Gamma_u\backslash \CY(\Gamma_u)_{\el,0}|$ &$\CF_u$ &$\Gamma_{\CF_u}$\\
\hline
\hline
$E_6$ &$1$ &$1$ &$1$ &$\phi_{1,36}$ &$1$\\
\hline
$E_6(a_1)$ &$1$ &$1$ &$1$ &$\phi_{6,25}$ &$1$\\
\hline
$E_6(a_3)$ &$1$ &$S_2$ &$4$ &$\phi_{30,15}$ &$C_2$\\
\hline
$D_4(a_1)$ &$(\bC^\times)^2$ &$S_3$ &$4$ &$\phi_{80,7}$ &$S_3$\\
\hline
\end{tabular}
\end{table}

\begin{table}[h]
\caption{Elliptic pairs for $E_7$\label{t:E7}}
\begin{tabular}{|c|c|c|c|c|c|}
\hline
Unipotent &$\Gamma_u^0$ &$A_u^{\ad}$ &$|\Gamma_u\backslash \CY(\Gamma_u)_{\el,0}|$&$\CF_u$ &$\Gamma_{\CF_u}$\\
\hline
\hline
$E_7$ &$1$ &$1$ &$1$ &$\phi_{1,63}$ &$1$\\
\hline
$E_7(a_1)$ &$1$ &$1$ &$1$ &$\phi_{7,46}$ &$1$\\
\hline
$E_7(a_2)$ &$1$ &$1$ &$1$ &$\phi_{27,37}$ &$1$\\
\hline
$E_7(a_3)$ &$1$ &$S_2$ &$4$ &$\phi_{56,30}$ &$C_2$\\
\hline
$E_7(a_4)$ &$1$ &$S_2$ &$4$ &$\phi_{189,22}$ &$1$\\
\hline
$E_7(a_5)$ &$1$ &$S_3$ &$8$ &$\phi_{315,16}$ &$S_3$\\
\hline
$E_6(a_1)$ &$\mathsf{SO}(2)$ &$S_2$ &$6$ &$\phi_{120,25}$ &$C_2$\\
\hline
$D_5(a_1)+A_1$ &$\mathsf{PGL}(2)$ &$1$ &$1$ &$\phi_{378,14}$ &$1$\\
\hline
$A_3+A_2+A_1$ &$\mathsf{PGL}(2)$ &$1$ &$1$ &$\phi_{210,10}$ &$1$\\
\hline
$A_4+A_1$ &$(\bC^\times)^2$ &$S_2$ &$3$ &$\phi_{512,11}$ &$C_2$\\
\hline
\end{tabular}
\end{table}

\begin{table}[h]
\caption{Elliptic pairs for $E_8$\label{t:E8}}
\begin{tabular}{|c|c|c|c|c|c|}
\hline
Unipotent &$\Gamma_u^0$ &$A_u$ &$|\Gamma_u\backslash \CY(\Gamma_u)_\el|$&$\CF_u$ &$\Gamma_{\CF_u}$\\
\hline
\hline
$E_8$ &$1$ &$1$ &$1$ &$\phi_{1,120}$ &$1$\\
\hline
$E_8(a_1)$ &$1$ &$1$ &$1$ &$\phi_{8,91}$ &$1$\\
\hline
$E_8(a_2)$ &$1$ &$1$ &$1$ &$\phi_{35,74}$ &$1$\\
\hline
$E_8(a_3)$ &$1$ &$S_2$ &$4$ &$\phi_{112,63}$ &$C_2$\\
\hline
$E_8(a_4)$ &$1$ &$S_2$ &$4$ &$\phi_{210,52}$ &$C_2$\\
\hline
$E_8(b_4)$ &$1$ &$S_2$ &$4$ &$\phi_{560,47}$ &$1$\\
\hline
$E_8(a_5)$ &$1$ &$S_2$ &$4$ &$\phi_{700,42}$ &$C_2$\\
\hline
$E_8(b_5)$ &$1$ &$S_3$ &$8$ &$\phi_{1400,37}$ &$S_3$\\
\hline
$E_8(a_6)$ &$1$ &$S_3$ &$8$ &$\phi_{1400,32}$ &$S_3$\\
\hline
$E_8(b_6)$ &$1$ &$S_3$ &$8$ &$\phi_{2240,28}$ &$C_2$\\
\hline
$E_8(a_7)$ &$1$ &$S_5$ &$39$ &$\phi_{4480,16}$ &$S_5$\\
\hline
$D_7(a_1)$ &$\mathsf{SO}(2)$ &$S_2$ &$6$ &$\phi_{3240,31}$ &$1$\\
\hline
$D_5+A_2$ &$\mathsf{SO}(2)$ &$S_2$ &$6$ &$\phi_{4536,23}$ &$1$\\
\hline
$E_6(a_1)+A_1$ &$\bC^\times$ &$S_2$ &$6$ &$\phi_{4096,26}$ &$C_2$\\
\hline
$D_7(a_2)$ &$\bC^\times$ &$S_2$ &$6$ &$\phi_{4200,24}$ &$C_2$\\
\hline
$A_6$ &$\mathsf{SL}(2)^2/\langle(-1,-1)\rangle$ &$1$ &$1$ &$\phi_{4200,21}$ &$1$\\
\hline
$A_4+A_2$ &$\mathsf{SL}(2)^2/\langle(-1,-1)\rangle$ &$1$ &$1$ &$\phi_{4536,13}$ &$1$\\
\hline
$A_4+2A_1$ &$\mathsf{GL}(2)$ &$S_2$ &$3$ &$\phi_{4200,12}$ &$C_2$\\
\hline
$D_4(a_1)+A_2$ &$\mathsf{PGL}(3)$ &$S_2$ &$1$ &$\phi_{2240,10}$ &$C_2$\\
\hline
\end{tabular}
\end{table}

\subsection{Distinguished $u$}
We first have the cases when $u$ is a distinguished orbit in $G^\vee$. For these cases $\Gamma_u=A_u$ and then $\CY(\Gamma_u)_\el=\CY(A_u).$ But in fact, since we will only look at the split adjoint $p$-adic group (rather than all of the forms inner to it), it is sufficient to consider $\CY(A_u^{\ad})$ for these cases. For exceptional groups, $A_u^{\ad}$ is one of $C_2$, $S_3$, $S_4$, $S_5$. For these groups we use the notation of Lusztig \cite{orange} to describe conjugacy classes and representations, see equivalently Carter \cite{Ca}.

If $A_u=C_2$, there are $4$ pairs $(s,h)$, and the corresponding virtual characters are
\begin{equation}
\begin{aligned}
\pi(u,1,1)&=\pi(u,1,\mathbf 1)+\pi(u,1,\epsilon);\\
\pi(u,1,g_2)&=\pi(u,1,\mathbf 1)-\pi(u,1,\epsilon);\\
\pi(u,g_2,1)&=\pi(u,g_2,\mathbf 1)+\pi(u,g_2,\epsilon);\\
\pi(u,g_2,g_2)&=\pi(u,g_2,\mathbf 1)-\pi(u,g_2,\epsilon).
\end{aligned}
\end{equation}

If $A_u=S_3$, there are $8$ pairs $(s,h)$ with corresponding virtual characters:
\begin{equation}
\begin{aligned}
\pi(u,1,1)&= \pi(u,1,\mathbf 1)+2\pi(1,r)+\pi(1,\epsilon);\\
\pi(u,1,g_2)&=\pi(1,\mathbf 1)-\pi(1,\epsilon);\\
\pi(u,1,g_3)&=\pi(1,\mathbf 1)-\pi(1,r)+\pi(1,\epsilon);\\
\pi(u,g_2,1)&=\pi(g_2,\mathbf 1)+ \pi(g_2,\epsilon);\\
\pi(u,g_2,g_2)&=\pi(g_2,\mathbf 1)- \pi(g_2,\epsilon);\\
\pi(u,g_3,1)&=\pi(u,g_3,\mathbf 1)+\pi(u,g_3,\theta)+\pi(u,g_3,\theta^2);\\
\pi(u,g_3,g_3)&=\pi(u,g_3,\mathbf 1)+\theta^2 \pi(u,g_3,\theta)+\theta \pi(u,g_3,\theta^2);\\
\pi(u,g_3,g_3^{-1})&=\pi(u,g_3,\mathbf 1)+\theta \pi(u,g_3,\theta)+\theta^2 \pi(u,g_3,\theta^2).
\end{aligned}
\end{equation}

If $A_u=S_4$, there are $21$ pairs  $(s,h)$ with corresponding virtual characters (below $\tau=g_2\cdot g_2'$ and $\gamma$ is conjugate to $g_2'$ in $S_4$, but not in $Z_{S_4}(g_2')=D_8$):
\begin{equation}
\begin{aligned}
\pi(u,1,1)&=\pi(u,1,\mathbf 1)+ 3 \pi(u,1,\lambda^1)+ 3 \pi(u,1,\lambda^2)+\pi(u,1,\lambda^3)+ 2 \pi(u,1,\sigma);\\
\pi(u,1,g_2)&=\pi(u,1,\mathbf 1)+  \pi(u,1,\lambda^1)- \pi(u,1,\lambda^2)-\pi(u,1,\lambda^3);\\
\pi(u,1,g_2')&=\pi(u,1,\mathbf 1)-  \pi(u,1,\lambda^1)- \pi(u,1,\lambda^2)+\pi(u,1,\lambda^3)+ 2 \pi(u,1,\sigma);\\
\pi(u,1,g_3)&=\pi(u,1,\mathbf 1)+ \pi(u,1,\lambda^3)- \pi(u,1,\sigma);\\
\pi(u,1,g_4)&=\pi(u,1,\mathbf 1)-  \pi(u,1,\lambda^1)+ \pi(u,1,\lambda^2)-\pi(u,1,\lambda^3);\\
\end{aligned}
\end{equation}
\begin{equation}
\begin{aligned}
\pi(u,g_2,1)&=\pi(u,g_2,\mathbf 1)+\pi(u,g_2,\epsilon')+\pi(u,g_2,\epsilon'')+\pi(u,g_2,\epsilon);\\
\pi(u,g_2,g_2)&=\pi(u,g_2,\mathbf 1)-\pi(u,g_2,\epsilon')+\pi(u,g_2,\epsilon'')-\pi(u,g_2,\epsilon);\\
\pi(u,g_2,\tau)&=\pi(u,g_2,\mathbf 1)+\pi(u,g_2,\epsilon')-\pi(u,g_2,\epsilon'')-\pi(u,g_2,\epsilon);\\
\pi(u,g_2,g_2')&=\pi(u,g_2,\mathbf 1)-\pi(u,g_2,\epsilon')-\pi(u,g_2,\epsilon'')+\pi(u,g_2,\epsilon);\\
\pi(u,g_2',1)&=\pi(u,g_2',\mathbf 1)+ 2 \pi(u,g_2',r)+\pi(u,g_2',\epsilon')+\pi(u,g_2',\epsilon'')+\pi(u,g_2',\epsilon);\\
\pi(u,g_2',g_2)&=\pi(u,g_2',\mathbf 1)-\pi(u,g_2',\epsilon')+\pi(u,g_2',\epsilon'')-\pi(u,g_2',\epsilon);\\
\pi(u,g_2',g_2')&=\pi(u,g_2',\mathbf 1)- 2 \pi(u,g_2',r)+\pi(u,g_2',\epsilon')+\pi(u,g_2',\epsilon'')+\pi(u,g_2',\epsilon);\\
\pi(u,g_2',g_4)&=\pi(u,g_2',\mathbf 1)-\pi(u,g_2',\epsilon')-\pi(u,g_2',\epsilon'')+\pi(u,g_2',\epsilon);\\
\pi(u,g_2',\gamma)&=\pi(u,g_2',\mathbf 1)+\pi(u,g_2',\epsilon')-\pi(u,g_2',\epsilon'')-\pi(u,g_2',\epsilon);\\
\end{aligned}
\end{equation}
\begin{equation}
\begin{aligned}
\pi(u,g_3,1)&=\pi(u,g_3,\mathbf 1)+\pi(u,g_3,\theta)+\pi(u,g_3,\theta^2);\\
\pi(u,g_3,g_3)&=\pi(u,g_3,\mathbf 1)+\theta^2 \pi(u,g_3,\theta)+\theta \pi(u,g_3,\theta^2);\\
\pi(u,g_3,g_3^{-1})&=\pi(u,g_3,\mathbf 1)+\theta \pi(u,g_3,\theta)+\theta^2\pi(u,g_3,\theta^2);\\
\pi(u,g_4,1)&=\pi(u,g_4,\mathbf 1)+\pi(u,g_4,i)+\pi(u,g_4,-1)+\pi(u,g_4,-i);\\
\pi(u,g_4,g_4)&=\pi(u,g_4,\mathbf 1)- i\pi(u,g_4,i)-\pi(u,g_4,-1)+i\pi(u,g_4,-i);\\
\pi(u,g_4,g_2')&=\pi(u,g_4,\mathbf 1)-\pi(u,g_4,i)+\pi(u,g_4,-1)-\pi(u,g_4,-i);\\
\pi(u,g_4,g_4^{-1})&=\pi(u,g_4,\mathbf 1)+i\pi(u,g_4,i)-\pi(u,g_4,-1)-i\pi(u,g_4,-i).
\end{aligned}
\end{equation}

\

If $A_u=S_5$, there are $39$ pairs $(s,h)$. We use Lusztig's notation, see \cite{orange} and \cite[p. 454]{Ca}. The representatives are:
\begin{align*}
s=1, \quad &Z_{A_u}(1)=S_5, &h:\ 1,\ g_2,\ g_2',\ g_3,\ g_4,\ g_5,\\
s=g_2,\quad  &Z_{A_u}(g_2)=\langle g_2\rangle\times S_3, &h:\ 1,\ g_2,\ \tau, \ g_3,\ g_2',\ g_6,\\
s=g_2',\quad &Z_{A_u}(g_2')=D_8, &h:\ 1,\ g_2,\ g_2',\ g_4,\ \gamma,\\
s=g_3,\quad &Z_{A_u}(g_3)=\langle g_2\rangle\times\langle g_3\rangle, &h:\ 1,\ g_2,\ g_3,\ g_3^{-1},\ g_6,\ g_6^{-1},\\
s=g_4,\quad &Z_{A_u}(g_4)=\langle g_4\rangle, &h:\ 1, g_4,\ g_2',\ g_4^{-1},\\
s=g_5,\quad &Z_{A_u}(g_5)=\langle g_5\rangle, &h:\ g_5^j,\ 0\le j<5,\\
s=g_6,\quad &Z_{A_u}(g_6)=\langle g_6\rangle, &h:\ g_6^j,\ 0\le j<6.\\
\end{align*}
The notation for the irreducible representations of $Z_{A_u}(s)$ are as in {\it loc. cit.}. We won't list all $39$ virtual character combinations. (The  ones of the form $\pi(u,g_2',h)$ have already been listed for $S_4$) But to illustrate the notation:
\begin{equation}
\begin{aligned}
\pi(u,1,1)&=\pi(u,1,\mathbf 1)+4\pi(u,1,\lambda^1)+5\pi(u,1,\nu)+6\pi(u,1,\lambda^2)+5\pi(u,1,\nu')+4\pi(u,1,\lambda^3)+\pi(u,1,\lambda^4),\\
\pi(u,1,g_2)&=\pi(u,1,\mathbf 1)+2\pi(u,1,\lambda^1)+\pi(u,1,\nu)-\pi(u,1,\nu')-2\pi(u,1,\lambda^3)-\pi(u,1,\lambda^4),\\
\\
\pi(u,g_2,1)&=\pi(u,g_2,\mathbf 1)+2\pi(u,g_2,r)+\pi(u,g_2,\epsilon)+\pi(u,g_2,-\mathbf 1)+2\pi(u,g_2,-r)+\pi(u,g_2,-\epsilon),\\
\pi(u,g_2',1)&=\pi(u,g_2',\mathbf 1)+ 2 \pi(u,g_2',r)+\pi(u,g_2',\epsilon')+\pi(u,g_2',\epsilon'')+\pi(u,g_2',\epsilon),\\
\pi(u,g_3,1)&=\pi(u,g_3,\mathbf 1)+\pi(u,g_3,\theta)+\pi(u,g_3,\theta^2)+\pi(u,g_3,-\mathbf 1)+\pi(u,g_3,-\theta)+\pi(u,g_3,-\theta^2),\\
\pi(u,g_4,1)&=\pi(u,g_4,\mathbf 1)+\pi(u,g_4,i)+\pi(u,g_4,-\mathbf 1)+\pi(u,g_4,-i),\\
\pi(u,g_5,1)&=\pi(u,g_5,\mathbf 1)+\pi(u,g_5,\zeta)+\pi(u,g_5,\zeta^2)+\pi(u,g_5,\zeta^3)+\pi(u,g_5,\zeta^4),\\
\pi(u,g_6,1)&=\pi(u,g_6,\mathbf 1)+\pi(u,g_6,\theta)+\pi(u,g_6,\theta^2)+\pi(u,g_6,-\mathbf 1)+\pi(u,g_6,-\theta)+\pi(u,g_6,-\theta^2).\\
\end{aligned}
\end{equation}

\

The same virtual combinations occur for the families of unipotent representations of finite reductive groups. In order to distinguish between the different families with the same $\Gamma_\CF$, we will use the notation 
\[\sigma(\mu,s,h)
\]
to denote the virtual character parameterised by the pair $(s,h)\in \Gamma_\CF^2$ for the family $\CF$ whose special unipotent representation (i.e., the irreducible unipotent character attached to $(1,\mathbf 1)$) is $\mu$.

\subsection{$\mathbf{\Gamma_u=\mathsf{PGL}(2,\bC)}$}\label{pgl2} This case appears for example when $u=B_3$ in $F_4$. Here 
\begin{equation}
\CY(\Gamma_u)=\text{ the }\Gamma_u\text{-orbit of } (s,h),
\end{equation}
where $s=\left(\begin{matrix} 1 &0\\0 &-1\end{matrix}\right)$ and $w=\left(\begin{matrix} 0 &1\\1 &0\end{matrix}\right)$ (modulo $Z(\mathsf{GL}(2))$). Since the component group of the centraliser in $\Gamma_u$ of $s$ is $C_2$, the only virtual character combination is
\begin{equation}
\pi(u,s,h)=\pi(u,s,\mathbf 1)-\pi(u,s,\epsilon).
\end{equation}

\subsection{$\mathbf{\Gamma_u=\mathsf O(2,\bC)}$} This case appears for example when $u=E_6(a_1)$ in $E_7$. Write \[\Gamma_u^0=\{z\mid |z|=1\}\subset \bC^\times\text{ and }\Gamma_u=\Gamma_u^0\cup \delta \Gamma_u^0,\] where $\delta^2=1$ and $\delta z \delta^{-1}=z^{-1}.$ 
The centre is $Z(\Gamma_u^0)=\{\pm 1\}$.

The centraliser of $\delta$ in $\Gamma_u$ is $\{\pm 1,\pm\delta\}\cong C_2\times C_2$. Denote the irreducible characters by $\{\mathbf 1,\epsilon_1,\epsilon_2,\epsilon\}$, where $\epsilon=\epsilon_1\otimes\epsilon_2$, and $\epsilon_1(-1)=-1$, $\epsilon_1(\delta)=1$, $\epsilon_2(-1)=1$, $\epsilon_2(\delta)=-1$. It is easy to see that the set of $\Gamma_u$-orbits on $\CY(\Gamma_u)_\el$ has representatives
\begin{equation}
\Gamma_u\backslash \CY(\Gamma_u)_\el:\quad (1,\delta),\ (-1,\delta),\ (\delta,1),\ (\delta,-1),\ (\delta,\delta),\ (\delta,-\delta).
\end{equation}
The component group of the centraliser of $1$ (also of $-1$) is $C_2$ and we denote the irreducible characters by $\{\mathbf 1,\epsilon\}$. The virtual character combinations are
\begin{equation}
\begin{aligned}
\pi(u,1,\delta)&=\pi(u,1,\mathbf 1)-\pi(u,1,\epsilon)\\
\pi(u,-1,\delta)&=\pi(u,-1,\mathbf 1)-\pi(u,-1,\epsilon)\\
\pi(u,\delta,1)&=\pi(u,\delta,\mathbf 1)+\pi(u,\delta,\epsilon_1)+\pi(u,\delta,\epsilon_2)+\pi(u,\delta,\epsilon)\\
\pi(u,\delta,-1)&=\pi(u,\delta,\mathbf 1)-\pi(u,\delta,\epsilon_1)+\pi(u,\delta,\epsilon_2)-\pi(u,\delta,\epsilon)\\
\pi(u,\delta,\delta)&=\pi(u,\delta,\mathbf 1)+\pi(u,\delta,\epsilon_1)-\pi(u,\delta,\epsilon_2)-\pi(u,\delta,\epsilon)\\
\pi(u,\delta,-\delta)&=\pi(u,\delta,\mathbf 1)-\pi(u,\delta,\epsilon_1)-\pi(u,\delta,\epsilon_2)+\pi(u,\delta,\epsilon)\\
\end{aligned}
\end{equation}

\subsection{$\mathbf{\Gamma_u=\mathsf O(n,\bC)}$, $\mathbf{n=3,4}$} This case appears when $G^\vee$ is a classical simple group, for example in the symplectic group.

Firstly, consider $\Gamma_u=\mathsf{O}(3)$. There are two conjugacy classes of semisimple elements of order $2$ that we need to consider
\[s_1=\left(\begin{matrix}1&0&0\\0&1&0\\0&0&-1\end{matrix}\right)\text{ and }s_2=\left(\begin{matrix}1&0&0\\0&-1&0\\0&0&-1\end{matrix}\right)
\]
with centralisers in $\Gamma_u$, $\mathsf O(2)\times \mathsf O(1)$ and $\mathsf O(1)\times \mathsf O(2)$, respectively. The component groups of $Z_{\Gamma_u}(s_i)$ are $C_2\times C_2$. Denote by $\delta$ a representative for the nontrivial element of the component group of $\mathsf O(2)$ and $\delta'$ for the nontrivial element of $\mathsf{O}(1)$. Notice that $\delta$ and $\delta'$ are $\Gamma_u$-conjugate to $s_1$, while $\delta\delta'$ is $\Gamma_u$-conjugate to $s_2$. Moreover $Z_{\Gamma_u}(s_i,\delta)=Z_{\Gamma}(s_i,\delta\delta')=C_2^3$, for $i=1,2$. There are 
\begin{equation}
\Gamma_u\backslash \CY(\Gamma_u)_\el: \quad (s_1,\delta),\ (s_1,\delta\delta'),\ (s_2,\delta),\ (s_2,\delta\delta'),
\end{equation}
where the first and the last are self-dual and the middle two pairs are dual to each other. The corresponding virtual character combinations are:
\begin{equation}
\begin{aligned}
\pi(u,s_1,\delta)=\sum_{\epsilon_1,\epsilon_2} \epsilon_1(\delta)\pi(u,s_1,\epsilon_1,\epsilon_2),\qquad\pi(u,s_1,\delta\delta')=\sum_{\epsilon_1,\epsilon_2,\epsilon_3} \epsilon_1(\delta)\epsilon_2(\delta')\pi(u,s_1,\epsilon_1,\epsilon_2),\\
\pi(u,s_2,\delta)=\sum_{\epsilon_1,\epsilon_2} \epsilon_2(\delta)\pi(u,s_2,\epsilon_1,\epsilon_2),\qquad\pi(u,s_2,\delta\delta')=\sum_{\epsilon_1,\epsilon_2} \epsilon_1(\delta')\epsilon_2(\delta)\pi(u,s_2,\epsilon_1,\epsilon_2),\\
\end{aligned}
\end{equation}
where $(\epsilon_1,\epsilon_2)=(\pm,\pm)$ are the characters of the $C_2\times C_2$'s.

\

If $\Gamma_u=\mathsf{O}(4)$, there is only one conjugacy class of semisimple elements of order $2$ which contributes to the elliptic pairs:
\[ s=\left(\begin{matrix} I_2&0\\0&-I_2\end{matrix}\right)
\]
with centraliser $\mathsf{O}(2)\times \mathsf{O}(2)$ and component group $C_2\times C_2$. Let $\Delta(\delta)$ denote the diagonal embedding of the element $\delta\in\mathsf{O}(2)$ from before. Notice that $\Delta(\delta)$ is $\Gamma_u$-conjugate to $s$ and $Z_{\Gamma_u}(s,\Delta(\delta))=(C_2\times C_2)^2$. There is only one (self-dual) elliptic pair $(s,\Delta(\delta))$ with virtual character
\begin{equation}
\pi(u,s,\Delta(\delta))=\pi(u,s,+,+)-\pi(u,s,-,+)-\pi(u,s,+,-)+-\pi(u,s,-,-).
\end{equation}

\subsection{$\mathbf{u=D_4(a_1)}$ in $\mathbf{E_6}$}\label{e6} This is a quasi-distinguished orbit where $\Gamma_u^0$ is a two-dimensional connected torus and $A_u=S_3$. If we realise 
\[
\Gamma_u^0=\{(x,y,z)\in (\bC^\times)^3\mid xyz=1\},
\]
then $\Gamma_u=\Gamma_u^0\ltimes S_3$, where $S_3$ acts on $\Gamma_u^0$ by permuting the coordinates. The centre of $\Gamma_u$ is
\[Z(\Gamma_u)=\langle (\zeta,\zeta,\zeta)\mid \zeta^3=1\rangle\cong C_3,
\]
which is in fact the centre of the simply-connected group $G^\vee=E_6$. Let $1,g_2,g_3$ denote representatives of the conjugacy classes in $A_u$. Then 
\[
Z_{\Gamma_u}(g_2)=\langle g_2, (x,x,x^{-2}) : x\in \bC^\times\rangle\cong \bC^\times \times C_2,\quad Z_{\Gamma_u}(g_3)=\langle Z(\Gamma_u), g_3\rangle \cong C_3\times C_3.
\]
One can verify that there are $12$ $\Gamma_u$-orbits on $\CY(\Gamma_u)_\el$ with representatives:
\begin{equation}
\Gamma_u\backslash \CY(\Gamma_u)_\el:\quad (z,g_3),\ (g_3, z),\ (g_3, z g_3),\ (g_3, z g_3^{-1}), \text{ where } z\in Z(\Gamma_u)=C_3. 
\end{equation}
Since we are only interested in the elliptic tempered representations for the split adjoint $p$-adic $E_6$, only $4$ pairs are relevant for the Langlands parameters, the ones with $z=1$, and the corresponding virtual character combinations are:
\begin{equation}
\begin{aligned}
\pi(u,1,g_3)&=\pi(u,1,\mathbf 1) - \pi(u,1,r)+\pi(u,1,\epsilon),\\
\pi(u,g_3,1)&=\pi(u,g_3,\mathbf 1) + \pi(u,g_3,\theta)+\pi(u,g_3,\theta^2),\\
\pi(u,g_3,g_3)&=\pi(u,g_3,\mathbf 1) +\theta^2 \pi(u,g_3,\theta)+\theta\pi(u,g_3,\theta^2),\\
\pi(u,g_3,g_3^{-1})&=\pi(u,g_3,\mathbf 1) +\theta \pi(u,g_3,\theta)+\theta^2\pi(u,g_3,\theta^2).\\
\end{aligned}
\end{equation}

\subsection{$\mathbf{u=A_4+A_1}$ in $\mathbf{E_7}$}\label{a4a1} This is a quasi-distinguished orbit  where $\Gamma_u^0$ is a two-dimensional torus and $A_u=C_2$. The quasi-distinguished property implies that $A_u$ acts on $\Gamma_u^0=(\bC^\times)^2$ by inversion. Let $\gamma$ denote a lift of the generator of $A_u$ in $\Gamma_u$. Hence 
\[ \gamma (x,y)\gamma^{-1}=(x^{-1},y^{-1}),\quad (x,y)\in \Gamma_u^0.
\]
Clearly $\delta^2\in Z(\Gamma_u)=\{(\pm 1,\pm 1)\}\cong C_2\times C_2$. The sequence $1\to \Gamma_0\to\Gamma\to A_u\to 1$ does not split \cite{Ad}. This fact can be seen also from \cite{Re2}, since the centraliser of $\gamma$ in $G^\vee=E_7$ is the pseudo-Levi of type $A_3+A_3+A_1$ and $A_{\delta u}=C_4$. This also means that if we set $Z(E_7)=\{\pm (1,1)\}\subset  Z(\Gamma_u),$ then we may set, without loss of generality, $\delta^2=(-1,1)$.

Then the centraliser $Z_{\Gamma_u}=Z(E_7)\times \langle\gamma\rangle\cong C_2\times C_4$. In particular, $Z_{\Gamma_u}/Z(E_7)=\langle \delta\rangle=C_4$.

One verifies easily that the orbits of elliptic pairs are:
\begin{equation}
\Gamma_u\backslash \CY(\Gamma_u)_\el:\quad (z,\delta),\ (\delta,z),\ (\delta,z\delta),\quad z\in Z(\Gamma_u)=C_2\times C_2.
\end{equation}

Then the virtual character combinations are:
\begin{equation}
\begin{aligned}
\pi(A_4A_1,z,\delta)&=\pi(A_4A_1,z,\mathbf 1)-\pi(A_4A_1,z,\epsilon),\\
\pi(A_4A_1,\delta,z)&=\sum_{\phi\in \widehat{C_2\times C_4}} \overline{\phi(z)} \pi(A_4A_1,\delta,\phi),\\
\pi(A_4A_1,\delta,z\delta)&=\sum_{\phi\in \widehat{C_2\times C_4}} \overline{\phi(z\delta)} \pi(A_4A_1,\delta,\phi).\\
\end{aligned}
\end{equation}
If we are interested only in the split $p$-adic adjoint form of $E_7$, then it is sufficient to consider:
\begin{equation}
\begin{aligned}
\pi(A_4A_1,1,\delta)&=\pi(A_4A_1,1,\mathbf 1)-\pi(A_4A_1,1,\epsilon),\\
\pi(A_4A_1,-1,\delta)&=\pi(A_4A_1,-1,\mathbf 1)-\pi(A_4A_1,-1,\epsilon),\\
\pi(A_4A_1,\delta,1)&=\pi(A_4A_1,\delta,\mathbf 1)+\pi(A_4A_1,\delta,i)+\pi(A_4A_1,\delta,-\mathbf 1)+\pi(A_4A_1,\delta,-i),\\
\pi(A_4A_1,\delta,\delta)&=\pi(A_4A_1,\delta,\mathbf 1)-i\pi(A_4A_1,\delta,i)-\pi(A_4A_1,\delta,-\mathbf 1)+i\pi(A_4A_1,\delta,-i),\\
\pi(A_4A_1,\delta,-1)&=\pi(A_4A_1,\delta,\mathbf 1)-\pi(A_4A_1,\delta,i)+\pi(A_4A_1,\delta,-\mathbf 1)-\pi(A_4A_1,\delta,-i),\\
\pi(A_4A_1,\delta,-\delta)&=\pi(A_4A_1,\delta,\mathbf 1)+i\pi(A_4A_1,\delta,i)-\pi(A_4A_1,\delta,-\mathbf 1)-i\pi(A_4A_1,\delta,-i).\\
\end{aligned}
\end{equation}
Here $-1=\delta^2$.

\subsection{$\mathbf{u=E_6(a_1)+A_1}$ or $\mathbf{D_7(a_2)}$ in $\mathbf{E_8}$}\label{s:E6(a1)A1} These are quasi-distinguished unipotent classes. The centraliser is 
\begin{equation}
\Gamma_u=\langle z,\delta\mid z\in \bC^\times,\ \delta z\delta^{-1}=z^{-1},\ \delta^2=-1\rangle,
\end{equation}
with $\Gamma_u=\bC^\times$ and $A_u=C_2$. The centre of $\Gamma_u$ is $Z(\Gamma_u)=\{\pm 1\}\cong C_2$. The centraliser of $\delta$ is $Z_{\Gamma_u}(\delta)=\{1,\delta,-1,-\delta\}\cong C_4$. There are $6$ virtual character combinations:
\begin{equation}
\begin{aligned}
\pi(u,1,\delta)&=\pi(u,1,\mathbf 1)-\pi(u,1,\epsilon),\\
\pi(u,-1,\delta)&=\pi(u,-1,\mathbf 1)-\pi(u,-1,\epsilon),\\
\pi(u,\delta,1)&=\pi(u,\delta,\mathbf 1)+\pi(u,\delta,i)+\pi(u,\delta,-\mathbf 1)+\pi(u,\delta,-i),\\
\pi(u,\delta,-1)&=\pi(u,\delta,\mathbf 1)-\pi(u,\delta,i)+\pi(u,\delta,-\mathbf 1)-\pi(u,\delta,-i),\\
\pi(u,\delta,1)&=\pi(u,\delta,\mathbf 1)-i\pi(u,\delta,i)-\pi(u,\delta,-\mathbf 1)+i\pi(u,\delta,-i),\\
\pi(u,\delta,1)&=\pi(u,\delta,\mathbf 1)+i\pi(u,\delta,i)-\pi(u,\delta,-\mathbf 1)-i\pi(u,\delta,-i).
\end{aligned}
\end{equation}

\subsection{$\mathbf{u=A_4+2A_1}$ in $\mathbf{E_8}$}\label{s:A_4+2A_1}\label{s:A4+2A1} This is not a quasi-distinguished unipotent class. The component group is $A_u=C_2=\langle\bar\delta\rangle.$ The reductive centraliser fits in the non-split short exact sequence
\begin{equation}
1\longrightarrow \mathsf{GL}(2)\longrightarrow \Gamma_u\longrightarrow A_u\longrightarrow 1,
\end{equation}
where $\delta^2=-1$ (so $\delta^4=1$) and $\delta x \delta^{-1}=(x^t)^{-1}$ for all $x\in \mathsf{GL}(2)$. The centre is $Z(\Gamma_u)=\{\pm 1\}$. The semisimple conjugacy classes in $\Gamma_u$ have representatives 
\[\left(\begin{matrix}a &0\\0&b\end{matrix}\right),\quad \delta,\quad \beta=\left(\begin{matrix}1 &1\\0&1\end{matrix}\right)\delta.
\]
Since we are interested in isolated semisimple elements, in particular torsion elements, we only need to consider 
\[s=\left(\begin{matrix}1 &0\\0&-1\end{matrix}\right)\text{ and }\delta, \text{ where }Z_{\Gamma_u}(s,\delta)=\langle s\rangle \times \langle \delta\rangle\cong C_2\times C_4.
\]
This is because there is no element in $\Gamma_u$ with finite centraliser, hence $\pm 1$ can't contribute to the elliptic pairs, and moreover, $\beta$ has order $8$, so it can't be isolated (there are no isolated elements of order $8$ in $E_8$). 

For $s$, we compute $Z_{\Gamma_u}(s)$ and find
\[
Z_{\Gamma_u}(s)=\langle x,\delta\mid x\in (\bC^\times)^2,\ \delta x\delta^{-1}=x^{-1}\rangle \text{ with }A_{\Gamma_u}(s)=\langle \bar\delta\rangle\cong C_2.
\]
Up to conjugacy, this gives one elliptic pair $(s,\delta)$, and
\[\pi(u,s,\delta)=\pi(u,s,\mathbf 1)-\pi(u,s,\epsilon).
\]

For $\delta$, we compute $Z_{\Gamma_u}(\delta)$ and find
\[ Z_{\Gamma_u}(\delta)=\langle x,\delta\mid x\in \mathsf{O}(2),\ \delta x\delta^{-1}=x,\ \delta^2=-1\rangle \text{ with }\CS_{\Gamma_u}(\delta)=\{1,s,\bar\delta,s\bar\delta\}\cong C_2\times C_2.
\]
Up to conjugacy, there are two elliptic pairs $(\delta,s)$ and $(\delta,s\delta)$ giving:
\begin{align*}
\pi(u,\delta,s)&=\pi(u,\delta,\mathbf 1)-\pi(u,\delta,\epsilon_1)+\pi(u,\delta,\epsilon_2)-\pi(u,\delta,\epsilon),\\
\pi(u,\delta,s\delta)&=\pi(u,\delta,\mathbf 1)-\pi(u,\delta,\epsilon_1)-\pi(u,\delta,\epsilon_2)+\pi(u,\delta,\epsilon),
\end{align*}
where $\epsilon=\epsilon_1\epsilon_2$, $\epsilon_1(s)=-1$, $\epsilon_2(\bar\delta)=1$, and $\epsilon_2(s)=1$, $\epsilon_2(\bar\delta)=-1$.

\subsection{$\mathbf{A_6}$ and $\mathbf{A_4+A_2}$ in $\mathbf{E_8}$}\label{s:A6} These are not quasi-distinguished classes. The centraliser is connected, see \cite{Ad}:
\[\Gamma_u=\mathsf{SL}(2)\times\mathsf{SL}(2)/\langle (-1,-1)\rangle.
\]
Denote by $\Delta(x)$, $x\in\mathsf{SL}(2)$, the image of the diagonal embedding of $x$. There is only one elliptic pair $(s,h)\sim (h,s)$, where 
\begin{equation}
s=\Delta\left(\begin{matrix}i &0\\0 &-i\end{matrix}\right),\quad h=\Delta\left(\begin{matrix}0 &1\\-1 &0\end{matrix}\right).
\end{equation}
The common centraliser is $Z_{\Gamma_u}(s,h)=\langle s\rangle\times\langle h\rangle \cong C_2\times C_2$. The relevant component group is $A_{\Gamma_u}(s)=\langle h\rangle\cong C_2$, and therefore the only virtual character combination is
\[\pi(u,s,h)=\pi(u,s,\mathbf 1)-\pi(u,s,\epsilon).
\] 

\subsection{$\mathbf{D_4(a_1)+A_2}$ in $\mathbf{E_8}$}\label{s:D4A2} This is not quasi-distinguished. The centraliser is 
\[\Gamma_u=\mathsf{PSL}(3)\rtimes \bZ/2,
\]
with component group $A_u=\langle\delta\rangle\cong C_2$, $\delta^2=1$, see \cite{Ad}. The action is $\delta x \delta^{-1}=(x^t)^{-1}$ for all $x\in \mathsf{PSL}(3)$. The centre is $Z(\Gamma_u)=\{1\}$. Denote 
\[
Z=Z(\mathsf{SL}(3)\cong C_3,
\]
and fix $\theta$ a primitive $3$-root of $1$. The centraliser of $\delta$ is 
\[Z_{\Gamma_u}(\delta)\cong \mathsf{SO}(3)\times \langle\delta\rangle
\]
and there is no semisimple element $h\in Z_{\Gamma_u}(\delta)$ such that $(\delta,h)$ is an elliptic pair. (Indeed, the centraliser of $\delta$ in $E_8$ is not a maximal quasi-Levi subgroup, i.e., $\delta$ is not isolated.) 

On the other hand, we know that the only maximal quasi-Levi subgroup in which $u$ is quasi-distinguished is $E_6+A_2$, which is the centraliser in $E_8$ of an element of order $3$. The elements of order $3$ in $\Gamma_u$ are all conjugate to
\[s=\left(\begin{matrix} 1&0&0\\0&\theta&0\\0&0&\theta^2\end{matrix}\right)\text{ mod } Z.
\]
We compute 
\begin{equation}
Z_{\Gamma_u}(s)=T_{su}\rtimes S_3,\text{ where } T_{su}=\{(x,y,z)\mid xyz=1\}/Z\cong (\bC^\times),
\end{equation}
$S_3$ acts by permutations and it is generated by 
\[
g_2=\left(\begin{matrix}0&1&0\\-1&0&0\\0&0&1\end{matrix}\right)\delta\quad\text{ and } g_3=\left(\begin{matrix}0&0&1\\1&0&0\\0&1&0\end{matrix}\right).
\]
There is only one $Z_{\Gamma_u}(s)$-conjugacy class in $Z_{\Gamma_u}(s)$, that of $g_3$, such that $Z_{\Gamma_u}(s,g_3)$ is finite. In fact, $Z_{\Gamma_u}(s,g_3)=\langle g_3\rangle$. (This is the same calculation as for $D_4(a_1)$ in $E_6$.) Hence, there is only one elliptic pair $(s,g_3)$, and notice that $s$ and $g_3$ are $\Gamma_u$-conjugate, hence this pair is self-dual. The corresponding virtual character combination is
\begin{equation}
\pi(u,s,g_3)=\pi(u,s,\mathbf 1)-\pi(u,s,r)+\pi(u,s,\epsilon).
\end{equation}

\section{$\mathbf{G_2}$}

There are two nilpotent orbits that occur. For the principal unipotent class:
\begin{equation}
\res^K_\un(\pi(G_2,1,1))=\sigma(\sgn,1,1)
\end{equation}
for all (maximal) parahoric subgroups $K$. 

\subsection{$\mathbf{G_2(a_1})$} The subregular  unipotent class is distinguished with component group $S_3$. We use the table for parahoric restriction from \cite{CO2}.

The restrictions to $K_0=G_2(q)$ are
\begin{equation}
\pi(G_2(a_1)), (s,h))=\sigma(\phi_{(2,1)},s,h)+\sigma(\phi_{(1,6)},1,1),
\end{equation}
for all $(s,h)\in \CY(S_3)$.

When the parahoric is of type $A_2$ or $A_1+\wti A_1$, the families in the finite reductive groups are all singletons and self-dual under the finite Fourier transform, so to simplify notation, we give the $K$-types, rather than the notation $\sigma(\CF,1,1)$.

For $K=A_2$, the restrictions are:
\begin{equation}
\begin{aligned}
&\pi(G_2(a_1),(1,1))=3 \epsilon+r, &\pi(G_2(a_1),(1,g_2))=\epsilon+r, \\
&\pi(G_2(a_1),(1,g_3))=r, &\pi(G_2(a_1),(g_2,1))=\epsilon+r, \\
&\pi(G_2(a_1),(g_3,1))=\pi(G_2(a_1),(g_3,g_3))=\pi(G_2(a_1),(g_3,g_3^{-1}))=r.
\end{aligned}
\end{equation} 

For $K=A_1+\wti A_1$, the restrictions are:
\begin{equation}
\begin{aligned}
&\pi(G_2(a_1),(1,1))=\epsilon\otimes\epsilon+3\epsilon\otimes 1 +1\otimes\epsilon,\\
&\pi(G_2(a_1),(1,g_2))=\pi(G_2(a_1),(g_2,1))=\epsilon\otimes\epsilon+\epsilon\otimes 1 +1\otimes\epsilon,\\
&\pi(G_2(a_1),(1,g_3))=\pi(G_2(a_1),(g_3,1))=\pi(G_2(a_1),(g_3,g_3))=\pi(G_2(a_1),(g_3,g_3^{-1}))=\epsilon\otimes\epsilon+1\otimes\epsilon.
\end{aligned}
\end{equation}

\section{$\mathbf {F_4}$} 

When we give the restrictions below, if a family of the finite group is a singleton, we give the $K$-type (rather than the $\sigma$ notation). To compute the families in the parahoric restrictions of discrete series representations, we refer to the tables in \cite{Re2}. For other representations we use \cite{BS} and \cite{Al}.

\

The $K_0$-restriction for the principal unipotent class is as before, $\pi(F_4,(1,1))=\phi_{1,24}$. 

\subsection{$\mathbf{F_4(a_1),F_4(a_2)}$} These are distinguished classes with component group $C_2$. 

The restrictions to $K_0=F_4(q)$ are

\begin{equation}
\pi(F_4(a_1),(s,h))=\sigma(\phi_{4,13},(s,h))+\phi_{1,24},
\end{equation}
for all $(s,h)\in \CY(C_2)$.

\

\begin{equation}
\pi(F_4(a_2),(s,h))=\phi_{9,10}+\sigma(\phi_{4,13},(s,h))+\phi_{1,24},
\end{equation}
for all $(s,h)\in \CY(C_2)$.

\subsection{$\mathbf{B_3}$} This is not a quasi-distinguished class. The centraliser is $\Gamma_u=\mathsf{PGL}(2,\bC)$, see section \ref{pgl2}. The centraliser of $s$ in $G^\vee=F_4(\bC)$ is a pseudo-Levi subgroup of type $B_4$ and the unipotent element $u$ lives in the subregular orbit $(711)$ of $B_4$ with component group $C_2$. To compute the $K$-structure of the corresponding irreducible tempered $G$-representations with Iwahori fixed vectors, we use Springer representations and induction, see for example \cite[Proposition (0.7)]{Re2}. We find that the restrictions to $K_0$ are
\begin{equation}
\begin{aligned}
\pi(B_3,s,\mathbf 1)&=\phi_{8,9}''+\phi_{4,13}+\phi_{2,16}''+\phi_{1,24},\\
\pi(B_3,s,\epsilon)&=\phi_{9,10},\\
\end{aligned}
\end{equation}
and therefore the virtual character combination is
\begin{equation}
\pi(B_3,s,h)=\phi_{8,9}'' + \sigma(\phi_{4,13},1,1)-\phi_{9,10}+\phi_{1,24},
\end{equation}
which is self-dual with respect to the finite Fourier transform.

\subsection{$\mathbf{F_4(a_3)}$} This is a distinguished class with component group $S_4$. Notice that in \cite[p. 64]{Re2}, there is a typo in the restriction for $[B_4(531),\epsilon'']$, the second $K_0$-type should be $\phi'_{(2,16)}$, rather than $\phi''_{(2,16)}.$ There is also a typo in \cite[p. 479]{Ca}, in the second $C_2$-family for $F_4(q)$, $\phi_{2,16}'$ and $\phi_{2,16}''$ should be swapped (compare to \cite{orange}).

\medskip

The restrictions to $K_0=F_4(q)$ are as follows:
\begin{equation}
\begin{aligned}
\pi(F_4(a_3),(1,1))&=\sigma(\phi_{12,4}, 1,1)+3 \sigma(\phi_{4,13},1,1)+ 4\phi'_{8,9}+\phi''_{8,9}+\phi_{9,10}+\phi_{1,24},\\
\pi(F_4(a_3),(1,g_2))&=\sigma(\phi_{12,4}, 1,g_2)+3 \sigma(\phi_{4,13},1,1)+ 2\phi'_{8,9}+\phi''_{8,9}+\phi_{9,10}+\phi_{1,24},\\
\pi(F_4(a_3),(g_2,1))&=\sigma(\phi_{12,4}, g_2,1)+3 \sigma(\phi_{4,13},1,1)+ 2\phi'_{8,9}+\phi''_{8,9}+\phi_{9,10}+\phi_{1,24},\\
\pi(F_4(a_3),(1,g_2'))&=\sigma(\phi_{12,4},1,g_2')+ \sigma(\phi_{4,13},1,1)+2\sigma(\phi_{4,13},1,g_2)+\phi''_{8,9}+\phi_{9,10}+\phi_{1,24},\\
\pi(F_4(a_3),(g_2',1))&=\sigma(\phi_{12,4},g_2',1)+ \sigma(\phi_{4,13},1,1)+2\sigma(\phi_{4,13},g_2,1)+\phi''_{8,9}+\phi_{9,10}+\phi_{1,24},\\
\pi(F_4(a_3),1,g_3)&=\sigma(\phi_{12,4},1,g_3)+\phi'_{8,9}+\phi''_{8,9}+\phi_{9,10}+\phi_{1,24},\\
\pi(F_4(a_3),g_3,1)&=\sigma(\phi_{12,4},g_3,1)+\phi'_{8,9}+\phi''_{8,9}+\phi_{9,10}+\phi_{1,24},\\
\pi(F_4(a_3),1,g_4)&=\sigma(\phi_{12,4},1,g_4)+\sigma(\phi_{4,13},1,g_2)+\phi''_{8,9}+\phi_{9,10}+\phi_{1,24},\\
\pi(F_4(a_3),g_4,1)&=\sigma(\phi_{12,4},g_4,1)+\sigma(\phi_{4,13},g_2,1)+\phi''_{8,9}+\phi_{9,10}+\phi_{1,24},\\
\end{aligned}
\end{equation}
\begin{equation}
\begin{aligned}
\pi(F_4(a_3),g_2,g_2)&=\sigma(\phi_{12,4},g_2,g_2)+\sigma(\phi_{4,13},1,1)+2\phi'_{8,9}+\phi''_{8,9}+\phi_{9,10}+\phi_{1,24},\\
\pi(F_4(a_3),(g_2,g_2'))&= \sigma(\phi_{12,4},g_2,g_2')+ (\sigma(\phi_{4,13},1,\mathbf 1)+\sigma(\phi_{4,13},g_2,\mathbf 1)) +\phi''_{8,9}+\phi_{9,10}+\phi_{1,24},\\
\pi(F_4(a_3),(g_2',g_2))&= \sigma(\phi_{12,4},g_2',g_2)+ (\sigma(\phi_{4,13},1,\mathbf 1)+\sigma(\phi_{4,13},g_2,\mathbf 1)) +\phi''_{8,9}+\phi_{9,10}+\phi_{1,24},\\
\pi(F_4(a_3),g_2,\tau)&=\sigma(\phi_{12,4},g_2,\tau)+\sigma(\phi_{4,13},1,1)+\phi''_{8,9}+\phi_{9,10}+\phi_{1,24},\\
\pi(F_4(a_3),g_2',g_2')&=\sigma(\phi_{12,4},g_2',g_2')+\sigma(\phi_{4,13},1,1)+2\sigma(\phi_{4,13},g_2,g_2)+\phi''_{8,9}+\phi_{9,10}+\phi_{1,24},\\
\pi(F_4(a_3),g_2',g_4)&=\sigma(\phi_{12,4},g_2',g_4)+\sigma(\phi_{4,13},1,g_2)+\phi''_{8,9}+\phi_{9,10}+\phi_{1,24},\\
\pi(F_4(a_3),g_4,g_2')&=\sigma(\phi_{12,4},g_4,g_2')+\sigma(\phi_{4,13},g_2,1)+\phi''_{8,9}+\phi_{9,10}+\phi_{1,24},\\
\pi(F_4(a_3),g_2',\gamma)&=\sigma(\phi_{12,4},g_2',\gamma)+\sigma(\phi_{4,13},1,g_2)+\sigma(\phi_{4,13},g_2,1)+\sigma(\phi_{4,13},g_2,g_2)\\&+\phi''_{8,9}+\phi_{9,10}+\phi_{1,24},\\
\end{aligned}
\end{equation}
\begin{equation}
\begin{aligned}
\pi(F_4(a_3),g_3,g_3)&=\sigma(\phi_{12,4},g_3,g_3)+\phi'_{8,9}+\phi''_{8,9}+\phi_{9,10}+\phi_{1,24},\\
\pi(F_4(a_3),g_3,g_3^{-1})&=\sigma(\phi_{12,4},g_3,g_3^{-1})+\phi'_{8,9}+\phi''_{8,9}+\phi_{9,10}+\phi_{1,24},\\
\pi(F_4(a_3),g_4,g_4)&=\sigma(\phi_{12,4},g_4,g_4)+\sigma(\phi_{4,13},g_2,g_2)+\phi''_{8,9}+\phi_{9,10}+\phi_{1,24},\\
\pi(F_4(a_3),g_4,g_4^{-1})&=\sigma(\phi_{12,4},g_4,g_4^{-1})+\sigma(\phi_{4,13},g_2,g_2)+\phi''_{8,9}+\phi_{9,10}+\phi_{1,24},\\
\end{aligned}
\end{equation}

\medskip

If a family $\CF$ is attached to the group $\Gamma_{\CF}=C_2$, then 
\begin{equation}\label{e:C2-stable}
\sigma(\CF,\mathbf 1)+\sigma(\CF,g_2,\mathbf 1)=\frac 12(\sigma(\CF,1,1)+\sigma(\CF,1,g_2)+\sigma(\CF,g_2,1)+\sigma(\CF,g_2,g_2)),
\end{equation}
and therefore, this expression is self-dual with respect to the finite $\FT$. This appears above in $\pi(F_4(a_3),(g_2,g_2'))$ for example.

\section{$\mathbf{E_6}$}

The $K_0$-restriction for principal class is as before, $\pi(E_6,(1,1))=\phi_{1,36}$. For the distinguished unipotent classes, we only consider $A_u^{\ad}$ since we are interested in the restrictions of the split $p$-adic group.

\subsection{$\mathbf{E_6(a_1)}$} This is a distinguished class with $A_u=1$. The only character restricts to $E_6(q)$ as:
\begin{equation}
\pi(E_6(a_1),1,1)=\phi_{6,25}+\phi_{1,36}
\end{equation}
(both representations in the right hand side form one-element families).

\subsection{$\mathbf{E_6(a_3)}$} This is a distinguished class with $A_u=C_2$. The restrictions to $E_6(q)$ are:

\begin{equation}
\begin{aligned}
\pi(E_6(a_3),1,1)&=\sigma(\phi_{30,15},1,1)+\phi_{20,20}+2\phi_{6,25}+\phi_{1,36},\\
\pi(E_6(a_3),1,g_2)&=\sigma(\phi_{30,15},1,g_2)+\phi_{20,20}+\phi_{1,36},\\
\pi(E_6(a_3),g_2,1)&=\sigma(\phi_{30,15},1,1)+\phi_{20,20}+\phi_{1,36},\\
\pi(E_6(a_3),g_2,g_2)&=\sigma(\phi_{30,15},g_2,g_2)+\phi_{20,20}+\phi_{1,36}.\\
\end{aligned}
\end{equation}

\subsection{$\mathbf{D_4(a_1)}$} This is the quasi-distinguished class with component group $S_3$, discussed in section \ref{e6}. We note that the centraliser of $g_3$ in $E_6$ is the pseudo-Levi subgroup of type $3A_2$. 

The $K_0$-restrictions of the virtual combinations are:
\begin{equation}
\begin{aligned}
\pi(D_4(a_1),1,g_3)&=\sigma(\phi_{80,7},1,g_3)  + \gamma(D_4(a_1)) ,\\
\pi(D_4(a_1),g_3,1)&=\sigma(\phi_{80,7},g_3,1)+  \gamma(D_4(a_1)),\\
\pi(D_4(a_1),g_3,g_3)&=\sigma(\phi_{80,7},g_3,g_3)+  \gamma(D_4(a_1)),\\
\pi(D_4(a_1),g_3,g_3^{-1})&=\sigma(\phi_{80,7},g_3,g_3^{-1}) +  \gamma(D_4(a_1)),\\
\end{aligned}
\end{equation}
where $\gamma(D_4(a_1))=(\sigma(\phi_{30,15},1,\mathbf 1)+\sigma(\phi_{30,15},g_2,\mathbf 1))+2\phi_{60,11}+\phi_{24,12}+2\phi_{20,20}+\phi_{1,36}$ is $\FT$-self-dual. 
Recall the observation about the self-duality of (\ref{e:C2-stable}).

\section{$\mathbf{E_7}$} The group of unramified characters of the split $p$-adic $E_7$ is $\bZ/2$, but since we are only interested in parahoric restrictions, we do not introduce the unramified twists in the notation below.
One can consult \cite{Re2}, where the precise information about which representations are isomorphic to their twist. For the same reason, the component groups $A_u$ below are given mod the centre $\bZ/2$. 

\subsection{$\mathbf{E_7(a_1),E_7(a_2)}$} These are distinguished classes with $A_u=1$. The restrictions to $K_0=E_7(q)$ are $\FT$-self-dual:

\begin{equation}
\pi(E_7(a_1), 1,1)=\phi_{7,46}+\phi_{1,63}.
\end{equation}
\begin{equation}
\pi(E_7(a_2), 1,1)=\phi_{27,37}+\phi_{7,46}+\phi_{1,63}.
\end{equation}

\subsection{$\mathbf{E_7(a_3)}$} This is a distinguished class with $A_u=S_2$. The $K_0$-restrictions  are:
\begin{equation}
\begin{aligned}
\pi(E_7(a_3),1,1)&=\sigma(\phi_{56,30},1,1)+\phi_{27,37}+2\phi_{7,46}+\phi_{1,63},\\
\pi(E_7(a_3),1,g_2)&=\sigma(\phi_{56,30},1,g_2)+\phi_{27,37}+\phi_{1,63},\\
\pi(E_7(a_3),g_2,1)&=\sigma(\phi_{56,30},g_2,1)+\phi_{27,37}+\phi_{1,63},\\
\pi(E_7(a_3),g_2,g_2)&=\sigma(\phi_{56,30},g_2,g_2)+\phi_{27,37}+\phi_{1,63}.\\
\end{aligned}
\end{equation}

\subsection{$\mathbf{E_7(a_4)}$} This is a distinguished orbit with $A_u=S_2$. The restrictions to $K=E_7(q)$ are:
\begin{equation}
\begin{aligned}
\pi(E_7(a_4),1,1)&=\phi_{189,22}+\sigma(\phi_{120,25},1,1)+\gamma(E_7(a_4)),\\
\pi(E_7(a_4),1,g_2)&=\phi_{189,22}+\sigma(\phi_{120,25},1,g_2)+\gamma(E_7(a_4)),\\
\pi(E_7(a_4),g_2,1)&=\phi_{189,22}+\sigma(\phi_{120,25},g_2,1)+\gamma(E_7(a_4)),\\
\pi(E_7(a_4),g_2,g_2)&=\phi_{189,22}+\sigma(\phi_{120,25},g_2,g_2)+\gamma(E_7(a_4)),\\
\end{aligned}
\end{equation}
where $\gamma(E_7(a_4))=\sigma(\phi_{56,30},1,1)+\phi_{21,36}+\phi_{27,37}+\phi_{7,46}+\phi_{1,63}$, an $\FT$-self-dual character.

\subsection{$E_7(a_5)$} This is a distinguished orbit with $A_u=S_3$. The restrictions to $K=E_7(q)$ are:

\begin{equation}
\begin{aligned}
\pi(E_7(a_5),1,1)&=\sigma(\phi_{315,16},1,1)+ 3 \sigma(\phi_{120,25},1,1)+\sigma(\phi_{56,30},1,1)+3\phi_{210,21}+3\phi_{27,37}+3\phi_{7,46}+\gamma(E_7(a_5)),\\
\pi(E_7(a_5),1,g_2)&=\sigma(\phi_{315,16},1,g_2)+  \sigma(\phi_{120,25},1,1)+\sigma(\phi_{56,30},1,g_2)+\phi_{210,21}+\phi_{27,37}+\phi_{7,46}+\gamma(E_7(a_5)),\\
\pi(E_7(a_5),g_2,1)&=\sigma(\phi_{315,16},g_2,1)+  \sigma(\phi_{120,25},1,1)+\sigma(\phi_{56,30},g_2,1)+\phi_{210,21}+\phi_{27,37}+\phi_{7,46}+\gamma(E_7(a_5)),\\
\pi(E_7(a_5),g_2,g_2)&=\sigma(\phi_{315,16},g_2,g_2)+  \sigma(\phi_{120,25},1,1)+\sigma(\phi_{56,30},g_2,g_2)+\phi_{210,21}+\phi_{27,37}+\phi_{7,46}+\gamma(E_7(a_5)),\\
\end{aligned}
\end{equation}
\begin{equation}
\begin{aligned}
\pi(E_7(a_5),1,g_3)&=\sigma(\phi_{315,16},1,g_3)+\gamma(E_7(a_5)),\\
\pi(E_7(a_5),g_3,1)&=\sigma(\phi_{315,16},g_3,1)+\gamma(E_7(a_5)),\\
\pi(E_7(a_5),g_3,g_3)&=\sigma(\phi_{315,16},g_3,g_3)+\gamma(E_7(a_5)),\\
\pi(E_7(a_5),g_3,g_3^{-1})&=\sigma(\phi_{315,16},g_3,g_3^{-1})+\gamma(E_7(a_5)),\\
\end{aligned}
\end{equation}
where we denoted $\gamma(E_7(a_5))=\sigma(\phi_{56,30},1,\mathbf 1)+\sigma(\phi_{56,30},g_2,\mathbf 1)+\phi_{105,21}+\phi_{168,21}+\phi_{189,22}+\phi_{21,36}+\phi_{27,37}+\phi_{1,63},$
an $\FT$-self-dual combination.

\subsection{$\mathbf{E_6(a_1)}$} This is quasi-distinguished with $\Gamma_u=\mathsf O(2)$. Since we only consider representations of the split $p$-adic form, the only virtual character combinations that contribute with their restriction to $K_0$ are as follows
\begin{equation}
\begin{aligned}
\pi(E_6(a_1),1,\delta)&=\sigma(\phi_{120,25},1,g_2)+\sigma(\phi_{56,30},1,g_2)+\phi_{21,36}+\phi_{1,63},\\
\pi(E_6(a_1),\delta,1)&=\sigma(\phi_{120,25},g_2,1)+\sigma(\phi_{56,30},g_2,1)+\phi_{21,36}+\phi_{1,63},\\
\pi(E_6(a_1),\delta,\delta)&=\sigma(\phi_{120,25},g_2,g_2)+\sigma(\phi_{56,30},g_2,g_2)+\phi_{21,36}+\phi_{1,63}.\\
\end{aligned}
\end{equation}

\subsection{$\mathbf{A_4+A_1}$}\label{s:A4A1} This is the quasi-distinguished orbit with $\Gamma_u^0=T_2$ and $A_u=C_2$ discussed in section \ref{a4a1}. The $K_0$-structure of the representations $\pi(u,\delta,\phi)$, $\phi\in\widehat C_4$ is in \cite[p. 71]{Re2} under the notation $[A_1A_3^2,\phi]$. We note that there is a typo, the first  $K_0$-type for $[A_1A_3^2,-1]$ should be $D_4(11,1)=D_4[r\epsilon_2]$ (in the notation of \cite{Ca}).

To compute $\pi(u,1,\mathbf 1)$ and $\pi(u,1,\epsilon)$, we use the tables in \cite{BS} for the nilpotent $A_4+A_1$ in $E_7$. Notice that \[\pi(u,1,\mathbf 1)+\pi(u,1,\epsilon)=\Ind_{A_4+A_1}^{E_7}(\sgn).
\] 

In the notation below, $-1=\delta^2$. This is a central element in $\Gamma_u$, but not in $G^\vee$.  The centraliser of $\delta^2$ in $G^\vee$ is a pseudo-Levi of type $D_6+A_1$. The representations $\pi(u,-1,\phi)$ have $K_0$-structures:
\begin{align*}
\pi(u,-1,\mathbf 1)&=\Ind_{D_6+A_1}^{E_7}((211\times 11+21^3\times 1+21^4\times 0+1^3\times 1^3+1^6\times 0)\otimes (11)),\\
\pi(u,-1,\epsilon)&=\Ind_{D_6+A_1}^{E_7}((1^4\times 2+21\times 1^3)\otimes (11)).
\end{align*}
We emphasise that $\pi(u,-1,\mathbf 1)$, $\pi(u,-1,\epsilon)$, contain $\phi_{512,11}$ and $\phi_{512,12}$, respectively, with multiplicity one.

\smallskip

The $K_0$-restrictions of the virtual character combinations are:
\begin{equation}
\begin{aligned}
\pi(A_4A_1,1,\delta)&=-\sigma(\phi_{512,11},1,g_2)+\sigma(\phi_{420,13},1,g_2)+\sigma(\phi_{405,15},1,g_2)+\sigma(\phi_{56,30},1,g_2)+\gamma(A_4A_1),\\
\pi(A_4A_1,\delta,1)&=\sigma(\phi_{512,11},g_2,1)+\sigma(\phi_{420,13},g_2,1)+\sigma(\phi_{405,15},g_2,1)+\sigma(\phi_{56,30},g_2,1)++\gamma(A_4A_1),\\
\pi(A_4A_1,\delta,\delta)&=-i\sigma(\phi_{512,11},g_2,g_2)+\sigma(\phi_{420,13},g_2,g_2)+\sigma(\phi_{405,15},g_2,g_2)+\sigma(\phi_{56,30},g_2,g_2)++\gamma(A_4A_1),\\
\pi(A_4A_1,\delta,-\delta)&=i\sigma(\phi_{512,11},g_2,g_2)+\sigma(\phi_{420,13},g_2,g_2)+\sigma(\phi_{405,15},g_2,g_2)+\sigma(\phi_{56,30},g_2,g_2)+\gamma(A_4A_1),\\
\pi(A_4A_1,-1,\delta)&=\sigma(\phi_{512,11},1,g_2)+\sigma(\phi_{420,13},1,g_2)+\sigma(\phi_{405,15},1,g_2)+\sigma(\phi_{56,30},1,g_2)+\gamma(A_4A_1),\\
\pi(A_4A_1,\delta,-1)&=-\sigma(\phi_{512,11},g_2,1)+\sigma(\phi_{420,13},g_2,1)+\sigma(\phi_{405,15},g_2,1)+\sigma(\phi_{56,30},g_2,1)+\gamma(A_4A_1),\\
\end{aligned}
\end{equation}
where
\begin{equation}
\begin{aligned}
\gamma(A_4A_1)=&\sigma(\phi_{315,16},(1,\mathbf 1)+(g_2,\mathbf 1)+(g_3,\mathbf 1))+\sigma(\phi_{120,25},(1,\mathbf 1)+(g_2,\mathbf 1))+\sigma(\phi_{56,30},(1,\mathbf 1)+(g_2,\mathbf 1))\\
&+\phi_{210,13}+\phi_{189,20}+\phi_{105,21}+2\phi_{168,21}+2\phi_{189,22}+\phi_{21,36}+2\phi_{27,37}+\phi_{1,63},\\
\end{aligned}
\end{equation}
is a self-dual combination with respect to $\FT$.

\subsection{$\mathbf{D_5(a_1)+A_1}$} This is not a quasi-distinguished orbit. The centraliser is $\Gamma_u=\mathsf{PGL}(2)$, so that $\CY(\Gamma_u)=\Gamma_u\cdot (s,h)$. The centraliser of $s$ in $G^\vee=E_7$ is the pseudo-Levi subgroup of type $D_6+A_1$, and $u$ corresponds to the unipotent orbit $(7311)\times (11)$ in $D_6+A_1$. The $K_0$-structure of the two irreducible elliptic tempered modules is
\begin{align*}
\pi(D_5(a_1)+A_1,s,\mathbf 1)&=\Ind_{D_6+A_1}^{E_7}((1112\times 1+1^5\times 1+1^4\times 1^2+11112\times 0+1^6\times 0)\otimes (11)),\\
\pi(D_5(a_1)+A_1,s,\epsilon)&=\Ind_{D_6+A_1}^{E_7}((1^4\times 2+1^5\times 1)\otimes (11)).
\end{align*}
Using \cite{Al}, we find that the restriction of the virtual combination is
\begin{equation}
\begin{aligned}
\pi(D_5(a_1)A_1,s,h)=&\phi_{378,14}+\frac 12\sigma(\phi_{315,16},(1,1)+(1,g_2)+(g_2,1)+(g_2,g_2))-\sigma(\phi_{405,15},1,1) \\&+ \frac 12 \sigma(\phi_{56,30},(g_2,1)+(1,g_2)+(g_2,g_2)-(1,1))\\
&+\sigma(\phi_{120,25},1,1)+\phi_{189,22}+\phi_{189,20}+\phi_{105,21}+2 \phi_{168,21}+3\phi_{27,37}+\phi_{1,63},
\end{aligned}
\end{equation}
which is indeed self-dual under the finite $\FT$.

\subsection{$\mathbf{A_3+A_2+A_1}$} This is not a quasi-distinguished class. The centraliser is $\Gamma_u=\mathsf{PGL}(2,\bC)$. The centraliser of the element $s$ is the pseudo-Levi subgroup of type $D_6+A_1$ and $u$ corresponds to the unipotent class $(5331)\times (2)$ in $D_6+A_1$. The $K_0$-structure of the two irreducible elliptic tempered modules is $\Ind_{D_6+A_1}^{E_7}(\pi_1\otimes (11))$ and $\Ind_{D_6+A_1}^{E_7}(\pi_2\otimes (11))$, respectively, where
\begin{align*}\pi_1|_{D_6}&=221\times 1+211\times 11+2111\times 1+1^4\times 11+1^3\times 1^3+1^5\times 1+2211\times 0+21^4\times 0+1^6\times 0,\\
\pi_2|_{D_6}&=222\times 0.
\end{align*}
We find the $K_0$-restriction is self-dual with respect to $\FT$:
\begin{equation}
\begin{aligned}
\pi(A_3A_2A_1,s,h)&=\phi_{210,10}+\sigma(\phi_{512,11},(1,1))+\sigma(\phi_{420,13},(1,\mathbf 1)+(g_2,\mathbf 1))+2\sigma(\phi_{405,15},(1,\mathbf 1)+(g_2,\mathbf 1))\\
&+\sigma(\phi_{315,16},3(1,\mathbf 1)+3(g_2,\mathbf 1)+(1,r)+2 (g_3,\mathbf 1))+2\sigma(\phi_{120,25},1,1)\\
&+\sigma(\phi_{120,25},(1,\mathbf 1)+(g_2,\mathbf 1))+3\sigma(\phi_{56,30},(1,\mathbf 1)+(g_2,\mathbf 1))+(...),
\end{aligned}
\end{equation}
where $(...)$ is a sum of one-element families.

\section{$\mathbf{E_8}$}
We give the $E_8(q)$-restrictions for the virtual character combinations of $E_8$. In place of $\phi_{d,e}$ we write $d_e$ for an irreducible $W(E_8)$-representation.

\subsection{Distinguished classes, $\mathbf{A_u=1}$} There are three orbits: $E_8$, $E(a_1)$, $E_8(a_2)$. The restrictions are all sums of one-element families, hence $\FT$-self-dual:

\begin{equation}
\begin{aligned}
\pi(E_8,(1,1))&=1_{120},\\
\pi(E_8(a_1),(1,1))&=8_{91}+1_{120},\\
\pi(E_8(a_2),(1,1))&=35_{74}+8_{91}+1_{120}.
\end{aligned}
\end{equation}

\subsection{Distinguished classes, $\mathbf{A_u=C_2}$} There are four classes: $E_8(a_3)$, $E_8(a_4)$, $E_8(b_4)$, $E_8(a_5)$. The restrictions are as follows:

\begin{equation}
\begin{aligned}
\pi(E_8(a_3),(1,1))&=\sigma(112_{63},(1,1))+ 2\cdot 8_{91}+ \gamma(E_8(a_3),\\
\pi(E_8(a_3),(1,g_2))&=\sigma(112_{63},(1,g_2))+  \gamma(E_8(a_3),\\
\pi(E_8(a_3),(g_2,1))&=\sigma(112_{63},(g_2,1))+  \gamma(E_8(a_3),\\
\pi(E_8(a_3),(g_2,g_2))&=\sigma(112_{63},(g_2,g_2))+ \gamma(E_8(a_3),\\
\end{aligned}
\end{equation}
where $ \gamma(E_8(a_3)=35_{74}+1_{120}$ is $\FT$-self-dual.

\

\begin{equation}
\begin{aligned}
\pi(E_8(a_4),(1,1))&=\sigma(210_{52},(1,1))+\sigma(112_{63},(1,1))+ 2\cdot 8_{91}+ 2\cdot 35_{74}+2\cdot 8_{91}+1_{120},\\
\pi(E_8(a_4),(1,g_2))&=\sigma(210_{52},(1,g_2))+\sigma(112_{63},(1,g_2))+ 1_{120},\\
\pi(E_8(a_4),(g_2,1))&=\sigma(210_{52},(g_2,1))+\sigma(112_{63},(g_2,1))+ 1_{120},\\
\pi(E_8(a_4),(g_2,g_2))&=\sigma(210_{52},(g_2,g_2))+\sigma(112_{63},(g_2,g_2))+ 1_{120}.\\
\end{aligned}
\end{equation}

\

\begin{equation}
\begin{aligned}
\pi(E_8(b_4),(1,1))&=560_{47}+\sigma(210_{52},(1,\mathbf 1)+(g_2,\mathbf 1))+  \gamma(E_8(b_4)),\\
\pi(E_8(b_4),(1,g_2))&=560_{47}+\sigma(210_{52},(1,\mathbf 1)-(g_2,\mathbf 1))  +  \gamma(E_8(b_4)),\\
\pi(E_8(b_4),(g_2,1))&=560_{47}+\sigma(210_{52},(1,\epsilon)+(g_2,\epsilon))+  \gamma(E_8(b_4)),\\
\pi(E_8(b_4),(g_2,g_2))&=560_{47}+\sigma(210_{52},(1,\epsilon)-(g_2,\epsilon))+ \gamma(E_8(b_4)),\\
\end{aligned}
\end{equation}
where $ \gamma(E_8(b_4))=\sigma(112_{63},(1,\mathbf 1)+(g_2,\mathbf 1))+ 35_{74}+8_{91}+1_{120}$ is $\FT$-self-dual. Notice that for a family $\CF$ such that $\Gamma_\CF=C_2$, the combinations $\sigma(1,\mathbf 1)+\sigma(g_2,\mathbf 1)$ and $\sigma(1,\epsilon)-\sigma(g_2,\epsilon)$ are self-dual, while \[\FT(\sigma(1,\mathbf 1)-\sigma(g_2,\mathbf 1))=\sigma(1,\epsilon)+\sigma(g_2,\epsilon).\]

\

\begin{equation}
\begin{aligned}
\pi(E_8(a_5),(1,1))&=\sigma(700_{42},(1,1))+\sigma(210_{52},(1,1))+35_{74}+\gamma(E_8(a_5)),\\
\pi(E_8(a_5),(1,g_2))&=\sigma(700_{42},(1,g_2))+\sigma(210_{52},(1,g_2))+\gamma(E_8(a_5)),\\
\pi(E_8(a_5),(g_2,1))&=\sigma(700_{42},(g_2,1))+\sigma(210_{52},(g_2,1))+\gamma(E_8(a_5)),\\
\pi(E_8(a_5),(g_2,g_2))&=\sigma(700_{42},(g_2,g_2))+\sigma(210_{52},(g_2,g_2))+\gamma(E_8(a_5)),\\
\end{aligned}
\end{equation}
where $\gamma(E_8(a_5))=560_{47}+\sigma(112_{63},(1,\mathbf 1)+(g_2,\mathbf 1))+8_{91}+1_{120}$ is $\FT$-self-dual.

\subsection{Distinguished classes, $\mathbf{A_u=S_3}$} There are three classes: $E_8(b_5)$, $E_8(a_6)$, $E_8(b_6)$.

\begin{equation}
\begin{aligned}
\pi(E_8(b_5),(1,1)&=\sigma(1400_{37},1,1)+3\cdot 567_{46}+3\sigma(112_{63},1,1)+ 3 \sigma(210_{52},1,1)+3\cdot 35_{74}\\&+3\cdot 8_{91} +\gamma(E_8(b_5)) \\
\pi(E_8(b_5),1,g_2)&=\sigma(1400_{37},1,g_2)+567_{46}+ \sigma(112_{63},1,g_2)+ \sigma(210_{52},1,1)+35_{74}+\gamma(E_8(b_5)) \\
\pi(E_8(b_5),1,g_3)&=\sigma(1400_{37},1,g_3)+  \gamma(E_8(b_5)) \\
\pi(E_8(b_5),g_2,1)&=\sigma(1400_{37},g_2,1)+567_{46} +\sigma(112_{63},g_2,1) + \sigma(210_{52},1,1)+35_{74} +\gamma(E_8(b_5)) \\
\pi(E_8(b_5),g_2,g_2)&= \sigma(1400_{37},g_2,g_2)+567_{46}+\sigma(112_{63},g_2,g_2) + \sigma(210_{52},1,1)+35_{74} +\gamma(E_8(b_5))\\
\pi(E_8(b_5),g_3,1)&=\sigma(1400_{37},g_3,1)  +\gamma(E_8(b_5))  \\
\pi(E_8(b_5),g_3,g_3)&= \sigma(1400_{37},g_3,g_3) +\gamma(E_8(b_5))  \\
\pi(E_8(b_5),g_3,g_3^{-1})&=\sigma(1400_{37},g_3,g_3^{-1}) +\gamma(E_8(b_5)) \\
\end{aligned}
\end{equation}
where $\gamma(E_8(b_5))=\sigma(112_{63},(1,\mathbf 1)+(g_2,\mathbf 1))+\sigma(700_{42},1,1)+560_{47}+35_{74}+1_{120}$ is $\FT$-self-dual. There is a typo in \cite{Re2}, for $[A_1E_7(a_2),+]$, there should be $-1400_{37}$ rather than $-35_{74}$.

\

\begin{equation}
\begin{aligned}
\pi(E_8(a_6),1,1)&=\sigma(1400_{32},1,1)+\sigma(1400_{37},1,1) + 3\sigma(210_{52},1,1)+3\sigma(112_{63},1,1)\\&+3\cdot 567_{46}+3\cdot 560_{47}
+3\cdot 35_{74}+3\cdot 8_{91}+\gamma(E_8(a_6)),\\
\pi(E_8(a_6),1,g_2)&=\sigma(1400_{32},1,g_2)+\sigma(1400_{37},1,g_2) +\sigma(210_{52},1,g_2)+\sigma(112_{63},1,g_2)\\&+567_{46} +560_{47}+ 35_{74}+ 8_{91}+\gamma(E_8(a_6),\\
\pi(E_8(a_6),1,g_3)&=\sigma(1400_{32},1,g_3)+\sigma(1400_{37},1,g_3)+\gamma(E_8(a_6),\\
\pi(E_8(a_6),g_2,1)&=\sigma(1400_{32},g_2,1)+\sigma(1400_{37},g_2,1) +\sigma(210_{52},g_2,1)+\sigma(112_{63},g_2,1)\\&+567_{46}+ 560_{47}+ 35_{74}+ 8_{91}+\gamma(E_8(a_6),\\
\pi(E_8(a_6),g_2,g_2)&=\sigma(1400_{32},g_2,g_2)+\sigma(1400_{37},g_2,g_2)+\sigma(210_{52},g_2,g_2)+\sigma(112_{63},g_2,g_2)\\&+567_{46}+ 560_{47}+ 35_{74}+ 8_{91} +\gamma(E_8(a_6),\\
\pi(E_8(a_6),g_3,1)&=\sigma(1400_{32},g_3,1)+\sigma(1400_{37},g_3,1) +\gamma(E_8(a_6),\\
\pi(E_8(a_6),g_3,g_3)&=\sigma(1400_{32},g_3,g_3)+\sigma(1400_{37},g_3,g_3) +\gamma(E_8(a_6),\\
\pi(E_8(a_6),g_3,g_3^{-1})&=\sigma(1400_{32},g_3,g_3^{-1})+\sigma(1400_{37},g_3,g_3^{-1}) +\gamma(E_8(a_6),\\
\end{aligned}
\end{equation}
where $\gamma(E_8(a_6)=\sigma(700_{42},(1,\mathbf 1)+(g_2,\mathbf 1))+\sigma(112_{63},(1,\mathbf 1)+(g_2,\mathbf 1))+1_{120}$ is $\FT$-self-dual.

\

For $E_8(b_6)$, we need the following observation about  $\FT$ for $S_3$-families, which follows immediately from the definition of $\FT$ in this case.

\begin{lemma}
Let $\CF$ be a family with $\Gamma_\CF=S_3$. The combinations 
\begin{align*}
\sigma(\CF, (1,\mathbf 1)+2(g_3,1)+(1,\epsilon)),\quad \sigma(\CF,(g_2,1)-(g_2,\epsilon)),\quad \sigma(\CF,(1,\mathbf 1)+(g_2,\mathbf 1)+(g_3,\mathbf 1)),\\
 \sigma(\CF,(1,\mathbf 1)+(g_2,\mathbf 1)+(1,r)),\quad  \sigma(\CF,(1,r)+\theta^2(g_3,\theta)+\theta(g_3,\theta^2)),\quad  \sigma(\CF,(1,r)+\theta(g_3,\theta)+\theta^2(g_3,\theta^2))
\end{align*}
are fixed by $\FT$, and
\begin{align*}
\FT(\sigma(\CF,(1,\mathbf 1)-(g_3,\mathbf 1)+(1,\epsilon)))&=\sigma(\CF, (1,r)+(g_3,\theta)+(g_3,\theta^2)),\\
\FT(\sigma(\CF,(1,\mathbf 1)-(1,\epsilon)))&=\sigma(\CF, (g_2,\mathbf 1)+(g_2,\epsilon)).\\
\end{align*}

\end{lemma}

The $K_0$-restrictions are:
\begin{equation}
\begin{aligned}
\pi(E_8(b_6),1,1)&=\sigma(2240_{28},1,1)+ \sigma(1400_{32},(1,\mathbf 1)+2(g_3,\mathbf 1)+(1,\epsilon))+\sigma(700_{42},1,1)\\&+\sigma(210_{52},1,1)+\gamma(E_8(b_6)),\\
\pi(E_8(b_6),1,g_2)&=\sigma(2240_{28},1,g_2)+ \sigma(1400_{32},(1,\mathbf 1)-(1,\epsilon))+\sigma(700_{42},1,g_2)\\
&+\sigma(210_{52},1,g_2)+\gamma(E_8(b_6)),\\
\pi(E_8(b_6),1,g_3)&=\sigma(2240_{28},1,1)+ \sigma(1400_{32},(1,\mathbf 1)-(g_3,\mathbf 1)+(1,\epsilon))+\sigma(700_{42},1,1)\\
&+\sigma(210_{52},1,1)+\gamma(E_8(b_6)),\\
\pi(E_8(b_6),g_2,1)&=\sigma(2240_{28},g_2,1)+ \sigma(1400_{32},(g_2,\mathbf 1)+(g_2,\epsilon))+\sigma(700_{42},g_2,1)\\
&+\sigma(210_{52},g_2,1)+\gamma(E_8(b_6)),\\
\pi(E_8(b_6),g_2,g_2)&=\sigma(2240_{28},g_2,g_2)+ \sigma(1400_{32},(g_2,\mathbf 1)-(g_2,\epsilon))+\sigma(700_{42},g_2,g_2)+\\&\sigma(210_{52},g_2,g_2)+\gamma(E_8(b_6)),\\
\pi(E_8(b_6),g_3,1)&=\sigma(2240_{28},1,1)+ \sigma(1400_{32},(1,r)+(g_3,\theta)+(g_3,\theta^2))+\sigma(700_{42},1,1)\\ 
&+\sigma(210_{52},1,1)+\gamma(E_8(b_6)),\\
\pi(E_8(b_6),g_3,g_3)&=\sigma(2240_{28},1,1)+ \sigma(1400_{32},(1,r)+\theta^2(g_3,\theta)+\theta(g_3,\theta^2))+\sigma(700_{42},1,1)&\\
&+\sigma(210_{52},1,1)+\gamma(E_8(b_6)),\\
\pi(E_8(b_6),g_3,g_3^{-1})&=\sigma(2240_{28},1,1)+ \sigma(1400_{32},(1,r)+\theta(g_3,\theta)+\theta^2(g_3,\theta^2))+\sigma(700_{42},1,1)&
\\&+\sigma(210_{52},1,1)+\gamma(E_8(b_6)),\\
\end{aligned}
\end{equation}
where $\gamma(E_8(b_6))=3240_{31}+567_{46}+2\cdot 560_{47}+\sigma(1400_{37},(1,\mathbf 1)+(g_2,\mathbf 1)+(g_3,\mathbf 1))+ 2 \sigma(112_{63},(1,\mathbf 1)+(g_2,\mathbf 1))+ \sigma(700_{42},(1,\mathbf 1)+(g_2,\mathbf 1)+35_{74}+8_{91}+1_{120}$, which is an $\FT$-stable combination.

\subsection{$\mathbf{E_8(a_7)}$} This is the distinguished class with $A_u=S_5$. We list the triples $(s,Z_{G^\vee}(s), u)$ such that, $s\in A_u$, $\operatorname{Ad}(G^\vee)u=E_8(a_7)$, and $u$ is distinguished in the maximal pseudo-Levi $Z_{G^\vee}(u)$:
\begin{equation}
\begin{aligned}
(1, E_8,E_8(a_7)),\quad (g_2, E_7\times A_1,E_7(a_5)\times (2)),\quad (g_2', D_8, (7531)),\quad (g_4, D_5\times A_3,(73)\times (4)),
\\ \quad (g_3, E_6\times A_2, E_6(a_3)\times (3)),\quad (g_5,A_4\times A_4,(5)\times (5))\quad, (g_6,A_5A_2A_1,(6)\times (3)\times (2)).
\end{aligned}
\end{equation}
To compute the relevant $K_0$-restrictions, we use a combination of \cite{Re2}, \cite{Al}, and \cite{BS}. For example, the restrictions of the discrete series parameterised by $g_2'$ are induced from the discrete series parameterised by the unipotent $(7531)$ in $D_8$. The latter have the following $D_8(q)$-structure:
\begin{align*}
\pi((7531),1)&=2211\times 11+21^3\times 1^3+1^4\times 1^4+1^5\times 1^3+21^4\times 1^2+1^6\times 1^2+221^3\times 1+21^5\times 1\\&+1^7\times 1+221^4\times 0+21^6\times 0+1^8\times 0,\\
\pi((7531),\epsilon'')&=22\times 1^4+1^5\times 21+1^6\times 1^2,\\
\pi((7531),\epsilon')&=2^4\times 0,
\end{align*}
which then need to be induced to $E_8$.

Denote $\mathbf v_2=(1,\mathbf 1)+(g_2,\mathbf 1)$ for an $S_2$-family and $\mathbf v_3=(1,\mathbf 1)+(g_2,\mathbf 1)+(g_3,\mathbf 1)$ for an $S_3$-family, which are both $\FT$-invariant.

The restrictions for the virtual combinations are quite lengthy, for example,
\begin{equation}
\begin{aligned}
\pi(E_8&(a_7),1,1)=\sigma(4480_{16},1,1)+10 \sigma(5600_{21},1,1)+ 5\cdot 4200_{21}+ 2835_{22}+5\cdot 6075_{22}+2\cdot 4536_{23}+\sigma(4200_{24},\mathbf v_2)\\&+10\sigma(4200_{24},1,1)+15 \sigma(2800_{25},1,1)+25\sigma(2268_{30},1,1)+\sigma(2268_{30},\mathbf v_2)+5\cdot 525_{36}+5\sigma(2240_{28},1,1)\\
&+2\sigma(2240_{28},\mathbf v_2)) +15 \sigma(4096_{26},1,1)+10 \sigma(1400_{32},1,1)+5\sigma(1400_{32},(1,\mathbf 1)+(g_2,\mathbf 1)+(1,r))\\
&+\sigma(1400_{32},\mathbf v_3)+22\cdot 3240_{31}+10\sigma(1400_{37},1,1)+15\sigma(1400_{37},(1,\mathbf 1)+(g_2,\mathbf 1)+(1,r))+2\sigma(1400_{37},\mathbf v_3)\\
&+11\sigma(700_{42},1,1)+7\sigma(700_{42},\mathbf v_2)+20\sigma(210_{52},1,1)+\sigma(210_{52},\mathbf v_2)\\
&+10\sigma(112_{63},1,1)+7\sigma(112_{63},\mathbf v_2)+35\cdot 567_{46}+22\cdot 560_{47}+11\cdot 35_{74}+5\cdot 8_{91}+1_{120}.
\end{aligned}
\end{equation}
This is self-dual with respect to $\FT$. There is a typo in the Fourier transform $S_5$-matrix \cite[p. 459]{Ca}, the entry for $\{(g_2,-r),(g_2,1)\}$ should be $-\frac 16$ (compare to \cite{orange}). We will not give all the terms in the expressions below.

\medskip

\begin{equation}
\begin{aligned}
\pi(E_8(a_7),1,g_2)&=\sigma(4480_{16},1,g_2)+\sigma(5600_{21},1,g_2)+\sigma(2268_{30},1,g_2)+\sigma(1400_{32},1,g_2)\\
&+3\sigma(1400_{37},1,g_2)+\sigma(112_{63},1,g_2)+\gamma(1,g_2),
\end{aligned}
\end{equation}
where $\gamma(1,g_2)=3\sigma(5600_{21},1,1)+3\cdot 4200_{21}+2835_{22}+4 \sigma(4200_{24},1,1)+\sigma(4200_{24},\mathbf v_2)+\dots+5\cdot 35_{74}+3\cdot 8_{91}+1_{120}$ is self-dual.

\medskip

\begin{equation}
\begin{aligned}
\pi(E_8&(a_7),1,g_2')=\sigma(4480_{16},1,g_2')+2\sigma(5600_{21},1,g_2)+2\sigma(2800_{25},1,g_2)+4\sigma(2268_{30},1,g_2)\\
&+2\sigma(1400_{32},1,g_2)+2\sigma(1400_{37},1,g_2)+2\sigma(210_{52},1,g_2)+2\sigma(112_{63},1,g_2)+\gamma(1,g_2'),
\end{aligned}
\end{equation}
where $\gamma(1,g_2')=4200_{21}+2835_{22}+6075_{22}+2\sigma(4200_{24},1,1)+\sigma(4200_{24},\mathbf v_2)+\sigma(2800_{25},1,1)+\dots+3\sigma(112_{63},\mathbf v_2)+3\cdot 35_{74}+8_{91}+1_{120}$ is self-dual.

\medskip

\begin{equation}
\begin{aligned}
\pi(E_8(a_7),1,g_3)&=\sigma(4480_{16},1,g_3)+\sigma(1400_{37},1,g_3)+\gamma(1,g_3),
\end{aligned}
\end{equation}
where $\gamma(1,g_3)=\sigma(5600_{21},1,1)+2\cdot 4200_{21}+2835_{22}+2\cdot 6075_{22}+2\cdot 4536_{23}+\sigma(1400_{37},\mathbf v_3+3((1,r)-(g_3,\mathbf 1)))+ \sigma(1400_{32},1,1)+\sigma(1400_{32},3\mathbf v_3+(g_3,\mathbf 1)-(1,r))+\dots+2\cdot 35_{74}+2\cdot 8_{91}+1_{120}$ is self-dual.

\medskip

\begin{equation}
\begin{aligned}
\pi(E_8(a_7),1,g_4)&=\sigma(4480_{16},1,g_4)+\sigma(2268_{30},1,g_2)+\sigma(2800_{25},1,g_2)+\sigma(210_{52},1,g_2)+\gamma(1,g_4),
\end{aligned}
\end{equation}
where $\gamma(1,g_4)=4200_{21}+2835_{22}+6075_{22}+\sigma(4200_{24},\mathbf v_2)+\sigma(2240_{28},1,1)+...+35_{74}+8_{91}+1_{120}$ is self-dual.

\medskip

\begin{equation}
\begin{aligned}
\pi(E_8(a_7), 1,g_5)&=\sigma(4480_{16},1,g_5)+\gamma(1,g_5),
\end{aligned}
\end{equation}
where $\gamma(1,g_5)=2\sigma(1400_{37},\mathbf v_3)+\sigma(1400_{32},\mathbf v_3)+2\sigma(2240_{28},\mathbf v_2)+\dots+35_{74}+1_{120}$ is $\FT$-self-dual.

\medskip

\begin{equation}
\begin{aligned}
\pi(E_8(a_7),1,g_6)&=\sigma(4480_{16},1,g_6)+\sigma(5600_{21},1,g_2)+\sigma(1400_{37},(1,g_3))+\sigma(1400_{32},1,g_2)\\&+\sigma(2268_{30},1,g_2)+\sigma(112_{63},1,g_2)+\gamma(g_2,g_3),
\end{aligned}
\end{equation}
where 
$\gamma(g_2,g_3)=\sigma(2268_{30},\mathbf v_2)+\sigma(1400_{32},\mathbf v_3)+2\sigma(1400_{37},\mathbf v_3)+\sigma(1400_{37},(1,\mathbf 1)+(1,r)+(g_2,\mathbf 1))+2\sigma(112_{63},\mathbf v_2)+\dots+\sigma(210_{52},\mathbf v_2)+1_{120}$ is self-dual.

\medskip

For $0\le j<2$, 
\begin{equation}
\begin{aligned}
\pi(E_8(a_7),g_2,g_2^j)&=\sigma(4480_{16},g_2,g_2^j)+\sigma(5600_{21},g_2,g_2^{j})+\sigma(2268_{30},g_2,g_2^j)+\sigma(1400_{32},g_2,g_2^j)\\
&+3\sigma(1400_{37},g_2,g_2^j)+\sigma(112_{63},g_2,g_2^j)+\gamma(1,g_2),
\end{aligned}
\end{equation}
where $\gamma(1,g_2)$ is above.

\medskip

\begin{equation}
\begin{aligned}
\pi(E_8(a_7),g_2,\tau)&=\sigma(4480_{16},g_2,\tau)+\sigma(5600_{21},1,g_2)+\sigma(5600_{21},g_2,1)+\sigma(1400_{37},1,g_2)+\sigma(1400_{37},g_2,1)\\
&+\sigma(1400_{32},1,g_2)+\sigma(1400_{32},g_2,1)+\sigma(2268_{30},1,g_2)+\sigma(2268_{30},g_2,1)\\
&+\sigma(112_{63},g_2,1)+\sigma(112_{63},1,g_2)+\gamma(g_2,g_2'),\\
\pi(E_8(a_7),g_2,g_2')&=\sigma(4480_{16},g_2,g_2')+\sigma(5600_{21},1,g_2)+\sigma(5600_{21},g_2,g_2)+\sigma(1400_{37},1,g_2)+\sigma(1400_{37},g_2,g_2)\\
&+\sigma(1400_{32},1,g_2)+\sigma(1400_{32},g_2,g_2)+\sigma(2268_{30},1,g_2)+\sigma(2268_{30},g_2,g_2)\\
&+\sigma(112_{63},g_2,g_2)+\sigma(112_{63},1,g_2)+\gamma(g_2,g_2'),\\
\end{aligned}
\end{equation}
where $\gamma(g_2,g_2')=\sigma(2240_{28},1,1)+\sigma(2240_{28},\mathbf v_2)+2\cdot 4536_{23}+4200_{21}+3\sigma(4096,{26},1,1)+\dots+3\sigma(112_{63},\mathbf v_2)+35_{74}+1_{120}$ is self-dual.

\medskip

\begin{equation}
\begin{aligned}
\pi(E_8(a_7),g_2,g_3)&=\sigma(4480_{16},g_2,g_3)+\sigma(2268_{30},g_2,1)+\sigma(1400_{32},g_2,1)+\sigma(5600_{21},g_2,1)\\
&+\sigma(1400_{37},1,g_3)+\sigma(112_{63},g_2,1)+\gamma(g_2,g_3),\\
\pi(E_8(a_7),g_2,g_6)&=\sigma(4480_{16},g_2,g_6)+\sigma(2268_{30},g_2,g_2)+\sigma(1400_{32},g_2,g_2)+\sigma(5600_{21},g_2,g_2)\\
&+\sigma(1400_{37},1,g_3)+\sigma(112_{63},g_2,g_2)+\gamma(g_2,g_3),\\
\end{aligned}
\end{equation}
where $\gamma(g_2,g_3)$ is as above.

\

For $0\le j<2$,
\begin{equation}
\begin{aligned}
\pi(E_8&(a_7),g_2',(g_2')^j)=\sigma(4480_{16},g_2',(g_2')^j)+2\sigma(5600_{21},g_2,g_2^j)+2\sigma(2800_{25},g_2,g_2^j)+4\sigma(2268_{30},g_2,g_2^j)\\
&+2\sigma(1400_{32},g_2,g_2^j)+2\sigma(1400_{37},g_2,g_2^j)+2\sigma(210_{52},g_2,g_2^j)+2\sigma(112_{63},g_2,g_2^j)+\gamma(1,g_2'),
\end{aligned}
\end{equation}
where $\gamma(1,g_2')$ is as above.

\begin{equation}
\begin{aligned}
\pi(E_8(a_7),g_2',g_2)&=\sigma(4480_{16},g_2',g_2)+\sigma(5600_{21},g_2,1)+\sigma(5600_{21},g_2,g_2)+\sigma(1400_{37},g_2,1)+\sigma(1400_{37},g_2,g_2)\\
&+\sigma(1400_{32},g_2,1)+\sigma(1400_{32},g_2,g_2)+\sigma(2268_{30},g_2,1)+\sigma(2268_{30},g_2,g_2)\\
&+\sigma(112_{63},g_2,g_2)+\sigma(112_{63},g_2,1)+\gamma(g_2,g_2'),\\
\end{aligned}
\end{equation}
where $\gamma(g_2,g_2')$ is as above.

\begin{equation}
\begin{aligned}
\pi(E_8(a_7),g_2',g_4)&=\sigma(4480_{16},g_2',g_4)+\sigma(2268_{30},1,g_2)+\sigma(2800_{25},1,g_2)+\sigma(210_{52},1,g_2)+\gamma(1,g_4),\\
pi(E_8(a_7),g_2',\gamma)&=\sigma(4480_{16},g_2',\gamma)+\sigma(2268_{30},1,g_2)+\sigma(2268_{30},g_2,1)+\sigma(2268_{30},g_2,g_2)+\sigma(2800_{25},1,g_2)\\
&+\sigma(2800_{25},g_2,1)+\sigma(2800_{25},g_2,g_2)+\sigma(210_{52},\mathbf v_2)+\gamma(1,g_4),
\end{aligned}
\end{equation}
where $\gamma(1,g_4)$ is as above.

\

For $0\le j<3$,
\begin{equation}
\begin{aligned}
\pi(E_8(a_7),g_3,g_3^j)&=\sigma(4480_{16},g_3,g_3^j)+\sigma(1400_{37},g_3,g_3^j)+\gamma(1,g_3),\\
\pi(E_8(a_7),g_3,g_2g_3^j)&=\sigma(4480_{16},g_3,g_2g_3^j)+\sigma(5600_{21},1,g_2)+\sigma(1400_{37},(g_3,g_3^j))+\sigma(1400_{32},1,g_2)\\&+\sigma(2268_{30},1,g_2)+\sigma(112_{63},1,g_2)+\gamma(g_2,g_3),
\end{aligned}
\end{equation}
where $\gamma(1,g_3)$ and  $\gamma(g_2,g_3)$ are the self-dual expression as above. 

\medskip

For $0\le j<4$,
\begin{equation}
\begin{aligned}
\pi(E_8(a_7),g_4,g_4^j)&=\sigma(4480_{16},g_4,g_4^j)+\sigma(2268_{30},g_2,g_2^j)+\sigma(2800_{25},g_2,g_2^j)+\sigma(210_{52},g_2,g_2^j)+\gamma(1,g_4),
\end{aligned}
\end{equation}
where $\gamma(1,g_4)$ is as above.

\medskip

For $0\le j<5$, 

\begin{equation}
\begin{aligned}
\pi(E_8(a_7), g_5,g_5^j)&=\sigma(4480_{16},g_5,g_5^j)+\gamma(1,g_5),
\end{aligned}
\end{equation}
where $\gamma(1,g_5)$ is as above.

\medskip

For $0\le j<6$,
\begin{equation}
\begin{aligned}
\pi(E_8(a_7),g_6,g_6^j)&=\sigma(4480_{16},g_6,g_6^j)+\sigma(5600_{21},g_2,g_2^j)+\sigma(1400_{37},(g_3,g_3^{j}))+\sigma(1400_{32},g_2,g_2^j)\\&+\sigma(2268_{30},g_2,g_2^j)+\sigma(112_{63},g_2,g_2^j)+\gamma(g_2,g_3),
\end{aligned}
\end{equation}
where $\gamma(g_2,g_3)$ is the self-dual expression as above. Recall that $g_6=g_2 g_3$, so $g_3=g_6^{-2}$.

\subsection{$\mathbf{D_7(a_1)}$ and $\mathbf{D_5+A_2}$} These are quasi-distinguished classes with $\Gamma_u=\mathsf O(2)$. Firstly, consider $D_7(a_1)$. The $8$ elliptic tempered representations that enter in the $6$ virtual combinations are $\pi(D_7(a_1),1,\mathbf 1)$ and $\pi(D_7(a_1),1,\epsilon)$ for which we use the tables in \cite{BS}, the $4$ discrete series representations corresponding to the pseudo-Levi $E_7+A_1$ and $u=E_7(a_3)+A_1$ (these are listed in \cite{Re2}, but there are two typos: $[A_1E_7(a_3),++]$ should also contain one $1008_{39}$ and $[A_1E_7(a_3),++]$ should also contain one $28_{68}$, this can be checked against the induction tables in \cite{Al}), and finally $\pi(D_7(a_1),-1,\mathbf 1)$ and $\pi(D_7(a_1),-1,\epsilon)$. The $K_0$-restrictions of these last two are as follows:

\begin{align*}
\pi(D_7(a_1),-1,\mathbf 1)&=\Ind_{D_8}^{E_8}( 21^5\times 1+1^7\times 1+1^6\times 1^2+21^6\times 0+1^8\times 0),\\
\pi(D_7(a_1),-1,\epsilon)&=\Ind_{D_8}^{E_8}(1^6\times 2+1^7\times 1),
\end{align*}
coming from the unipotent class $(11,3,1,1)$ in $D_8$.

The restrictions of the virtual character combinations are:

\begin{equation}
\begin{aligned}
\pi(D_7(a_1),1,\delta)&=(3240_{31})+\sigma(1400_{32},(1,\mathbf 1)-(g_2,\mathbf 1)+(1,r))+\sigma(1400_{37},(1,\mathbf 1)+(g_2,\mathbf 1)+(1,r))\\
&+\sigma(700_{42},1,1)+\sigma(700_{42},1,g_2)+\sigma(210_{52},1,1)+\sigma(210_{52},1,g_2)+\gamma(D_7(a_1)),\\
\pi(D_7(a_1),-1,\delta)&=(3240_{31})-\sigma(1400_{32},(1,\mathbf 1)+(1,r)-(g_2,\mathbf 1))+\sigma(1400_{37},(1,\mathbf 1)+(g_2,\mathbf 1)+(1,r))\\
&+\sigma(700_{42},1,1)+\sigma(700_{42},1,g_2)+\sigma(210_{52},(g_2,\mathbf 1)+(1,\epsilon))+\sigma(112_{63},1,g_2)\\
&+\sigma(112_{63},g_2,1)+\sigma(112_{63},g_2,g_2)+560_{47}+2\cdot 35_{74}+1_{120},\\
\pi(D_7(a_1),\delta,1)&=(3240_{31})+\sigma(1400_{32},(1,r)+(1,\epsilon)+(g_2,\epsilon))+\sigma(1400_{37},(1,\mathbf 1)+(g_2,\mathbf 1)+(1,r))\\
&+\sigma(700_{42},1,1)+\sigma(700_{42},g_2,1)+\sigma(210_{52},1,1)+\sigma(210_{52},g_2,1)+\gamma(D_7(a_1)),\\
\pi(D_7(a_1),\delta,\delta)&=(3240_{31})+\sigma(1400_{32},(1,r)+(1,\epsilon)-(g_2,\epsilon))+\sigma(1400_{37},(1,\mathbf 1)+(g_2,\mathbf 1)+(1,r))\\
&+\sigma(700_{42},1,1)+\sigma(700_{42},g_2,g_2)+\sigma(210_{52},1,1)+\sigma(210_{52},g_2,g_2)+\gamma(D_7(a_1)),\\
\pi(D_7(a_1),\delta,-1)&=(3240_{31})-\sigma(1400_{32},(1,r)+(1,\epsilon)+(g_2,\epsilon))+\sigma(1400_{37},(1,\mathbf 1)+(g_2,\mathbf 1)+(1,r))\\
&+\sigma(700_{42},1,1)+\sigma(700_{42},g_2,1)+\sigma(210_{52},(1,\mathbf 1)-(g_2,\epsilon))+\sigma(112_{63},1,g_2)\\
&+\sigma(112_{63},g_2,1)+\sigma(112_{63},g_2,g_2)+560_{47}+2\cdot 35_{74}+1_{120},\\
\pi(D_7(a_1),\delta,-\delta)&=(3240_{31})-\sigma(1400_{32},(1,r)+(1,\epsilon)-(g_2,\epsilon))+\sigma(1400_{37},(1,\mathbf 1)+(g_2,\mathbf 1)+(1,r))\\
&+\sigma(700_{42},1,1)+\sigma(700_{42},g_2,g_2)+\sigma(210_{52},(1,\mathbf 1)+(g_2,\epsilon))+\sigma(112_{63},1,g_2)\\
&+\sigma(112_{63},g_2,1)+\sigma(112_{63},g_2,g_2)+560_{47}+2\cdot 35_{74}+1_{120}.\\
\end{aligned}
\end{equation}
Here $\gamma(D_7(a_1))=2\cdot 567_{46}+3\cdot 560_{47}+2\sigma(112_{63},1,1)+2\sigma(112_{63},\mathbf v_2)+2\cdot 35_{74}+2\cdot 8_{91}+1_{120}$ is self-dual.

\

Next, it's $D_5+A_2$. The $8$ elliptic tempered representations that enter in the $6$ virtual combinations are $\pi(D_5A_2,1,\mathbf 1)$ and $\pi(D_5A_2,1,\epsilon)$ for which we use the tables in \cite{BS}, the $4$ discrete series representations corresponding to the pseudo-Levi $E_7+A_1$ and $u=E_7(a_4)+A_1$ (these are listed in \cite{Re2}), and finally $\pi(D_5A_2,-1,\mathbf 1)$ and $\pi(D_5A_2,-1,\epsilon)$. The $K_0$-restrictions of these last two are as follows:

\begin{align*}
\pi(D_5A_2,-1,\mathbf 1)&=\Ind_{D_8}^{E_8}( 221^3\times 1+21^5\times 1+1^7\times 1+1^6\times 11+21^4\times 11+1^5\times 1^3+221^4\times 0\\&+21^6\times 0+1^8\times 0),\\
\pi(D_5A_2,-1,\epsilon)&=\Ind_{D_8}^{E_8}(2^311\times 0),
\end{align*}
coming from the unipotent class $(9,3,3,1)$ in $D_8$.

The restrictions of the virtual character combinations are:

\begin{equation}
\begin{aligned}
\pi(D_5A_2,1,\delta)&=4536_{23}+\sigma(4200_{24},(1,\mathbf 1)-(g_2,\mathbf 1))+\sigma(2240_{28},1,g_2)+\sigma(700_{42},1,g_2)+\gamma(D_5A_2),\\
\pi(D_5A_2,-1,\delta)&=4536_{23}+\sigma(4200_{24},(1,\mathbf 1)-(g_2,\mathbf 1))+\sigma(2240_{28},(g_2,\mathbf 1)+(1,\epsilon))-\sigma(700_{42},1,g_2)+\gamma(D_5A_2),\\
\pi(D_5A_2,\delta,1)&=4536_{23}+\sigma(4200_{24},(1,\epsilon)+(g_2,\epsilon))+\sigma(2240_{28},g_2,1)+\sigma(700_{42},g_2,1)+\gamma(D_5A_2),\\
\pi(D_5A_2,\delta,\delta)&=4536_{23}+\sigma(4200_{24},(1,\epsilon)-(g_2,\epsilon))+\sigma(2240_{28},g_2,g_2)+\sigma(700_{42},g_2,g_2)+\gamma(D_5A_2),\\
\pi(D_5A_2,\delta,-\delta)&=4536_{23}+\sigma(4200_{24},(1,\epsilon)-(g_2,\epsilon))-\sigma(2240_{28},1,g_2)-\sigma(700_{42},1,g_2)+\gamma(D_5A_2),\\
\pi(D_5A_2,\delta,-1)&=4536_{23}+\sigma(4200_{24},(1,\epsilon)+(g_2,\epsilon))+\sigma(2240_{28},(1,\mathbf 1)-(g_2,\epsilon))-\sigma(700_{42},g_2,1)+\gamma(D_5A_2),\\
\end{aligned}
\end{equation}
where $\gamma(D_5A_2)=\sigma(2240_{28},\mathbf v_2)+\sigma(4096_{26},1,1)+\dots+3\sigma(112_{63},\mathbf v_2)+2\cdot 35_{74}+8_{91}+1_{120}$ is self-dual.

\subsection{$\mathbf{E_6(a_1)+A_1}$ and $\mathbf{D_7(a_2)}$} These are quasi-distinguished classes with \[\Gamma_u=\langle z,\delta\mid z\in \bC^\times,\ \delta z\delta^{-1}=z^{-1},\ \delta^2=-1\rangle.\]

First, we consider $E_6(a_1)+A_1$. With the notation as in section \ref{s:E6(a1)A1}, the relevant centralisers are $Z_{G^\vee}(-1)=E_7\times A_1$ and $Z_{G^\vee}(\delta)=A_7+A_1$ and $u$ corresponds to $E_6(a_1)+A_1$ and the principal unipotent class, respectively. The restrictions of the discrete series representations parameterised by $A_7+A_1$ are in \cite{Re2}. For the other elliptic tempered modules 
\[\pi(E_6(a_1)A_1,1,\mathbf 1),\ \pi(E_6(a_1)A_1,1,\epsilon)
\]
we use the tables in \cite{BS}, and for 
\[ \pi(E_6(a_1)A_1,-1,\phi)=\Ind_{E_7\times A_1}^{E_8}(\pi(u,1,\phi)\otimes (11)),\quad \phi=\mathbf 1,\epsilon,
\]
the $K_0$-structure of $\pi(E_6(a_1)+A_1,\phi)$ in $E_7$ and the induction tables of \cite{Al}.

The virtual character combinations are:

\begin{equation}
\begin{aligned}
\pi(E_6&(a_1)A_1,\pm 1,\delta)=\mp\sigma(4096_{26},1,g_2)+\sigma(2240_{28},1,g_2)+\sigma(2268_{30},1,g_2)+\sigma(1400_{32},1,g_2)+\sigma(210_{52},1,g_2)\\&+\sigma(112_{62},1,g_2)+\gamma(E_6(a_1)A_1),\\
\pi(E_6&(a_1)A_1,\delta,\pm1)=\pm\sigma(4096_{26},g_2,1)+\sigma(2240_{28},g_2,1)+\sigma(2268_{30},g_2,1)+\sigma(1400_{32},g_2,1)+\sigma(210_{52},g_2,1)\\
\pi(E_6&(a_1)A_1,\delta,\pm\delta)=\mp i\sigma(4096_{26},g_2,g_2)+\sigma(2240_{28},g_2,g_2)+\sigma(2268_{30},g_2,g_2)+\sigma(1400_{32},g_2,g_2)\\
&+\sigma(210_{52},g_2,g_2)+\sigma(112_{62},g_2,g_2)+\gamma(E_6(a_1)A_1),\\
\end{aligned}
\end{equation}
where $\gamma(E_6(a_1)A_1)=\sigma(1400_{37},\mathbf v_3)+\sigma(700_{42},\mathbf v_2)+\sigma(112_{63},\mathbf v_2)+560_{47}+35_{74}+1_{120}$ is self-dual.

\

Now, take $D_7(a_2)$. With the notation as in section \ref{s:E6(a1)A1}, the relevant centralisers are $Z_{G^\vee}(-1)=D_8$ and $Z_{G^\vee}(\delta)=D_5+A_3$ and $u$ corresponds to $(9,5,1,1)$ and the principal unipotent class, respectively. The restrictions of the discrete series representations parameterised by $D_5+A_3$ are in \cite{Re2}. For the other elliptic tempered modules 
\[\pi(D_7(a_2),1,\mathbf 1),\ \pi(D_7(a_2),1,\epsilon)
\]
we use the tables in \cite{BS}, and for
\begin{align*}
\pi(D_7(a_2),-1,\mathbf 1)&=\Ind_{D_8}^{E_8}(21^3\times 11+1^6\times 11+1^5\times 1^3+21^5\times 1+1^7\times 1+21^6\times 0+1^8\times 0),\\
\pi(D_7(a_2),-1,\epsilon)&=\Ind_{D_8}^{E_8}(1^5\times 21+1^6\times 2+1^6\times 11+1^7\times 1)
\end{align*}
the induction tables of \cite{Al}.

The virtual character combinations are:

\begin{equation}
\begin{aligned}
\pi(D_7&(a_2),\pm 1,\delta)=\sigma(4200_{24},1,g_2)\mp\sigma(4096_{26},1,g_2)+\sigma(2268_{30},1,g_2)+\sigma(112_{63},1,g_2)+\gamma(D_7(a_2)),\\
\pi(D_7&(a_2),\delta,\pm 1)=\sigma(4200_{24},g_2,1)\pm\sigma(4096_{26},g_2,1)+\sigma(2268_{30},g_2,1)+\sigma(112_{63},g_2,1)+\gamma(D_7(a_2)),\\
\pi(D_7&(a_2),\delta,\pm \delta)=\sigma(4200_{24},g_2,g_2)\mp i\sigma(4096_{26},g_2,g_2)+\sigma(2268_{30},g_2,g_2)+\sigma(112_{63},g_2,g_2)+\gamma(D_7(a_2)),\\
\end{aligned}
\end{equation}
where $\gamma(D_7(a_2))=\sigma(2240_{28},\mathbf v_2)+3240_{31}+\sigma(1400_{37},\mathbf v_3)+\sigma(700_{42},1,1)+\sigma(700_{42},\mathbf v_2)+\sigma(210_{52},\mathbf v_2)+\sigma(112_{63},\mathbf v_2)+560_{47}+35_{74}+1_{120}$ is self-dual.

\subsection{$\mathbf{A_6}$ and $\mathbf{A_4+A_2}$} The centraliser is as in section \ref{s:A6}. Firstly, consider $u=A_6$. The centraliser $Z_{G^\vee}(s)=D_8$ and the corresponding quasi-distinguished unipotent class is $(7711)$ in $D_8$. The two elliptic representations that occur in the virtual character combination have $K_0$-structure induced from the corresponding two elliptic tempered representations parameterised by $(7711)$ in $D_8$ and with Springer local systems. We get
\begin{equation}
\begin{aligned}
\pi(A_6,s,\mathbf 1)&=\Ind_{D_8}^{E_8}(21^3\times 1^3+1^5\times 1^3+1^4\times 1^4+21^4\times 11+1^6\times 11+21^5\times 1+1^7\times 1+21^6\times 0+1^8\times 0),\\
\pi(A_6,s,\epsilon)&=\Ind_{D_8}^{E_8}(211\times 1^4+1^5\times 1^3+1^5\times 21+1^6\times 1^2+1^6\times 2+1^7\times 1).
\end{aligned}
\end{equation}
We find
\begin{equation}
\begin{aligned}
\pi(A_6,s,h)&=4200_{21}-6075_{22}+\sigma(2800_{25},\mathbf v_2)+\sigma(2240_{28},\mathbf v_2)+\sigma(2240_{28},(g_2,\mathbf 1)-(1,\epsilon))+\dots+35_{74}+1_{120},
\end{aligned}
\end{equation}
an $\FT$-self-dual combination. 

\

For the unipotent class $A_4+A_2$, again the relevant centraliser is $Z_{G^\vee}(s)=D_8$ and the corresponding quasi-distinguished class is $(5533)$ in $D_8$. As before, the two elliptic representations that occur in the virtual character combination have $K_0$-structure induced from the corresponding two elliptic tempered representations parameterised by $(5533)$ in $D_8$ and with Springer local systems. We get
\begin{equation}
\begin{aligned}
\pi(A_4+A_2,s,\mathbf 1)&=\Ind_{D_8}^{E_8}(221\times 21+222\times 11+221\times 1^3+211\times 211+\dots+2\cdot 1^7\times 1+21^6\times 0+1^8\times 0),\\
\pi(A_4+A_2,s,\epsilon)&=\Ind_{D_8}^{E_8}(211\times 22+221\times 1^3+211\times 1^4+\dots).
\end{aligned}
\end{equation}
Denote
\[\mathbf u_5=2(1,\mathbf 1)+(1,\lambda^1)+2(g_3,\mathbf 1)+(g_3,\epsilon)+2(g_2',\mathbf 1)+(g_5,\mathbf 1)+2(g_2,\mathbf 1)+(g_6,\mathbf 1)+(g_2,r)+2\cdot (g_4,\mathbf 1),
\]
for the $S_5$-family. One verifies that $\FT(\mathbf u_5)=\mathbf u_5$.

We find
\begin{equation}
\begin{aligned}
\pi(A_4+A_2,s,h)&=4536_{13}-2835_{14}-6075_{14}+\sigma(5600_{15},\mathbf v_2)+4200_{15}+2100_{20}+\sigma(4480_16,\mathbf u_5)+\dots,
\end{aligned}
\end{equation}
an $\FT$-self-dual combination. 

\subsection{$\mathbf{A_4+2A_1}$}\label{s:A42A1} The centraliser $\Gamma_u$ is discussed in \ref{s:A4+2A1}.
For the $S_5$-family, denote
\begin{align*}
\mathbf v_5'=&2 (1,\mathbf 1)-(1,\lambda^2)+(1,\nu)+(1,\nu')+(g_5,\mathbf 1)+2(g_6,\mathbf 1)+(g_3,\mathbf 1)-(g_3,\epsilon)+2(g_2',\mathbf 1)+2(g_2,\mathbf 1),\\
\mathbf v_5''=&(1,\mathbf 1)+(1\nu)+(g_5,\mathbf 1)+(g_4,\mathbf 1)+(g_4,-\mathbf 1)+2(g_6,\mathbf 1)+(g_3,\mathbf 1)-(g_3,\epsilon)+2(g_2',\mathbf 1)\\
&+(g_2',\epsilon')+(g_2',\epsilon'')+(g_2',r)+2(g_2,\mathbf 1),\\
\mathbf v_5=&(1,\mathbf 1)+(1\nu)+(g_5,\mathbf 1)+(g_4,\mathbf 1)-(g_4,-\mathbf 1)+2(g_6,\mathbf 1)+(g_3,\mathbf 1)-(g_3,\epsilon)+2(g_2',\mathbf 1)\\
&+(g_2',\epsilon')+(g_2',\epsilon'')-(g_2',r)+2(g_2,\mathbf 1).\\
\end{align*}
Then
\[\FT(\mathbf v_5')=\mathbf v_5''\text{ and } \FT(\mathbf v_5)=\mathbf v_5.
\]

With the notation as in section \ref{s:A4+2A1}, the three virtual character combinations are:
\begin{align*}
\pi(A_4+2A_1,s,\delta)&=\pi(A_4+2A_1,s,\mathbf 1)-\pi(A_4+2A_1,s,\epsilon),\\
\pi(A_4+2A_1,\delta,s)&=\pi(A_4+2A_1,\delta,\mathbf 1)-\pi(A_4+2A_1,\delta,\epsilon_1)+\pi(A_4+2A_1,\delta,\epsilon_2)-\pi(A_4+2A_1,\delta,\epsilon),\\
\pi(A_4+2A_1,\delta,s)&=\pi(A_4+2A_1,\delta,\mathbf 1)-\pi(A_4+2A_1,\delta,\epsilon_1)-\pi(A_4+2A_1,\delta,\epsilon_2)+\pi(A_4+2A_1,\delta,\epsilon).
\end{align*}
The centralisers are $Z_{G^\vee}(s)=E_7+A_1$ and $Z_{G^\vee}(\delta)=D_5+A_3$. In $E_7+A_1$, the unipotent class is $A_4+A_1$ in $E_7$ and the principal class in $A_1$. This gives the $K_0$-restriction:
\[\pi(A_4+2A_1,s,\delta)=\Ind_{E_7\times A_1}^{E_8}((\pi_{E_7}(A_4+A_1,1,\mathbf 1)-\pi_{E_7}(A_4+A_1,1,\epsilon))\otimes (11)),
\]
and the right hand side is known from the $E_7$ calculations (and the induction tables). 

In $D_5+A_3$, the unipotent class is $(5311)$ in $D_5$ and the principal class in $A_3$. The restrictions of the first two elliptic representations are induced from corresponding elliptic representations for $D_5$ attached to $(5311)$ and the two Springer local systems:
\begin{equation}
\begin{aligned}
\pi(A_4+2A_1,\delta,\mathbf 1)&=\Ind_{D_5\times A_3}^{E_8}(211\times 1+1^3\times 1^2+1^4\times 1+21^3\times 0+1^5\times 0),\\
\pi(A_4+2A_1,\delta,\epsilon_1)&=\Ind_{D_5\times A_3}^{E_8}(1^3\times 2+1^4\times 1).
\end{aligned}
\end{equation}
The remaining two elliptic representations correspond to the affine Hecke algebra $\CH(E_8/D_4)$ of type $F_4$ (see [p. 80]\cite{Re2})  with parameters:
\begin{equation}
\CH(E_8/D_4):\quad 1 \relbar 1\relbar 1 \Rightarrow 4\relbar 4
\end{equation}
with semisimple parameter $\delta$ and furthermore to the graded affine Hecke algebra of type $B_3\times A_1$ with parameters on $B_3$:
\begin{equation}
\bH(B_3,2):\quad 1\relbar 1\Rightarrow 2.
\end{equation}
At the level of $\bH(B_3,2)$, the $W(B_3)$-structure of the two elliptic modules (limits of discrete series) is
\begin{equation}
\pi(B_3,(5311),+)=0\times 12+1\times 2\text{ and } \pi(B_3,(5311),-)=0\times 3.
\end{equation}
Inducing to $F_4$, we find the corresponding $K_0$-structures are:
\begin{equation}
\begin{aligned}
\pi(A_4+2A_1,\delta,\epsilon_2)&=\phi_{4,8}+\phi_{16,5}+\phi_{9,10}+\phi_{4,7}''+\phi_{8,9}''+\phi_{8,3}''+\phi_{9,6}'+\phi_{2,16}'',\\
\pi(A_4+2A_1,\delta,\epsilon)&=\phi_{1,12}''+\phi_{9,6}'+\phi_{2,4}''.
\end{aligned}
\end{equation}

Putting all of these calculations together we find:
\begin{equation}
\begin{aligned}
\pi(A_4&+2A_1,s,\delta)=-\sigma(4200_{12},1,g_2)-\sigma(5600_{15},(g_2,\mathbf 1)+(1,\epsilon))+\sigma(4480_{16},\mathbf v_5')+\sigma(5600_{21},1,g_2)\\
&+\sigma(2240_{28},1,g_2)+\sigma(2268_{30},1,g_2)+2\sigma(1400_{32},(1,\mathbf 1)+(g_2,\mathbf 1)-(1,\epsilon))+\sigma(700_{42},1,g_2)\\
&+\sigma(210_{52},1,g_2)+\gamma(A_4+2A_1),\\
\pi(A_4&+2A_1,\delta,s)=-\sigma(4200_{12},g_2,1)-\sigma(5600_{15},(1,\mathbf 1)-(g_2,\epsilon))+\sigma(4480_{16},\mathbf v_5'')+\sigma(5600_{21},g_2,1)\\
&+\sigma(2240_{28},g_2,1)+\sigma(2268_{30},g_2,1)+\sigma(1400_{32},(1,\mathbf 1)+3(g_2,\mathbf 1)-(1,\epsilon)+(g_2,\epsilon))\\&
+\sigma(700_{42},g_2,1)+\sigma(210_{52},g_2,1)+\gamma(A_4+2A_1),\\
\pi(A_4&+2A_1,\delta,s\delta)=-\sigma(4200_{12},g_2,g_2)-\sigma(5600_{15},(1,\mathbf 1)+(g_2,\epsilon))+\sigma(4480_{16},\mathbf v_5)+\sigma(5600_{21},g_2,g_2)\\
&+\sigma(2240_{28},g_2,g_2)+\sigma(2268_{30},g_2,g_2)+\sigma(1400_{32},(1,\mathbf 1)+3(g_2,\mathbf 1)-(1,\epsilon)-(g_2,\epsilon))\\&+\sigma(700_{42},g_2,g_2)+\sigma(210_{52},g_2,g_2)+\gamma(A_4+2A_1),\\
\end{aligned}
\end{equation}
where $\gamma(A_4+2A_1)=4200_{15}+4536_{13}+2100_{20}+\sigma(5600_{21},\mathbf v_2)+\dots+3\cdot 35_{74}+1_{120}$ is self-dual.

\subsection{$\mathbf{D_4(a_1)+A_2}$}. This is the non-quasi-distinguished unipotent class discussed in section \ref{s:D4A2}. The $K_0$-restriction of the virtual character combination is
\[
\pi(D_4(a_1)+A_2,s,g_3)=\Ind_{E_6\times A_2}^{E_8}(\pi_{E_6}(D_4(a_1)+A_2,1,g_3)),
\]
and the inducing data in the right hand side is known from the calculations in $E_6$. Set 
\begin{align*}\mathbf u_5'=&3(1,\mathbf 1)+(1,\lambda^1)+2(1,\nu)+2(g_5,\mathbf 1)+(g_4,\mathbf 1)+4(g_6,\mathbf 1)+3 (g_3,\mathbf 1)-(g_3,\epsilon)\\&+3(g_2',\mathbf 1)+(g_2',\epsilon'')+4(g_2,\mathbf 1)+(g_2,\epsilon)+(g_2,r),
\end{align*}
which has the property $\FT(\mathbf u_5')=\mathbf u_5'$.
We find

\begin{equation}
\begin{aligned}
\pi(D_4&(a_1)+A_2,s,g_3)=\sigma(2240_{10},1,1)+\sigma(2800_{13},1,1)+\sigma(5600_{15},(g_2,\mathbf 1)-(1,\epsilon))+3\cdot 4200_{15}\\&-\sigma(4200_{12},(g_2,\mathbf 1)-(1,\epsilon))+3\cdot 2835_{14}+\sigma(4096_{11},1,1)-2100_{20}+\sigma(4480_{16},\mathbf u'_5)+\dots,
\end{aligned}
\end{equation}
an $\FT$-self-dual expression.

\ifx\undefined\bysame
\newcommand{\bysame}{\leavevmode\hbox to3em{\hrulefill}\,}
\fi

\end{document}